\documentclass[11pt]{article}
\usepackage{amsmath,fullpage,amssymb,amsthm,hyperref,enumerate}
\usepackage{tikz,color}

\newtheorem{theorem}{Theorem}[section]

\newtheorem{lemma}[theorem]{Lemma}
\newtheorem{corollary}[theorem]{Corollary}
   
\newtheorem{conjecture}[theorem]{Conjecture}  
\newtheorem{problem}[theorem]{Problem}  

\theoremstyle{remark}
  
\DeclareMathOperator\des{des}
\DeclareMathOperator\exc{exc}
\DeclareMathOperator\asc{asc}
\DeclareMathOperator\cdes{cdes}
\DeclareMathOperator\casc{casc}

\DeclareMathOperator\Var{Var}
\DeclareMathOperator\plat{plat}
\DeclareMathOperator\lea{lea}
\DeclareMathOperator\emp{emp}
\newcommand{\QQ}{\overline{\mathcal{Q}}}
\newcommand{\QQQ}{\widetilde{\mathcal{Q}}}
\newcommand{\Q}{\mathcal{Q}}
\newcommand{\T}{\mathcal{T}}
\newcommand{\I}{\mathcal{I}}
\newcommand{\In}{\mathcal{J}}
\newcommand{\U}{\mathcal{U}}
\newcommand{\A}{\mathcal{A}}
\renewcommand{\S}{\mathcal{S}}

\renewcommand{\d}{d}
\newcommand{\oQ}{\overline{Q}}
\newcommand{\oP}{\overline{P}}
\newcommand{\oG}{G} 
\DeclareMathOperator\Seq{\textsc{Seq}}

\date{}

\title{Descents on quasi-Stirling permutations}

\author{Sergi Elizalde}

\begin{document}

\maketitle

\begin{abstract}
Stirling permutations were introduced by Gessel and Stanley~\cite{gessel_stirling_1978}, who used their
enumeration by the number of descents to give a combinatorial interpretation of certain polynomials 
related to Stirling numbers.

Quasi-Stirling permutations, which can be viewed as labeled noncrossing matchings, were introduced by Archer et al.\ \cite{archer_pattern_2019} as a natural extension of Stirling permutations. 
Janson's correspondence~\cite{janson_plane_2008} between Stirling permutations and labeled increasing plane trees
extends to a bijection between quasi-Stirling permutations and the same set of trees without the increasing restriction.

Archer et al.~\cite{archer_pattern_2019} posed the problem of enumerating quasi-Stirling permutations by the number of descents, and conjectured
that there are $(n+1)^{n-1}$ such permutations of size $n$ having the maximum number of descents. In this paper we prove their conjecture, and we give the generating function for quasi-Stirling permutations by the number of descents, expressed as a compositional inverse of the generating function of Eulerian polynomials. We also find the analogue for quasi-Stirling permutations of the main result from~\cite{gessel_stirling_1978}. We prove that the distribution of descents on these permutations is asymptotically normal, and that the roots of the corresponding quasi-Stirling polynomials are all real, in analogy to B\'ona's results for Stirling permutations~\cite{bona_real_2008}.

Finally, we generalize our results to a one-parameter family of permutations that extends $k$-Stirling permutations, and we refine them by also keeping track of the number of ascents and the number of plateaus.
\end{abstract}

\section{Introduction}

\subsection{Stirling permutations}

In 1978, Gessel and Stanley~\cite{gessel_stirling_1978} introduced the set $\Q_n$ of Stirling permutations. They are defined as those permutations $\pi_1\pi_2\dots\pi_{2n}$ of the multiset $\{1,1,2,2,\dots,n,n\}$ satisfying that, if $i<j<k$ and $\pi_i=\pi_k$, then $\pi_j>\pi_i$.
In pattern avoidance terminology, we can describe this condition as avoiding the pattern $212$.

In general, given two sequences of positive integers $\pi=\pi_1\pi_2\dots\pi_r$ and $\sigma=\sigma_1\sigma_2\dots\sigma_s$, we say that $\pi$ avoids $\sigma$ if there is no subsequence
$\pi_{i_1}\pi_{i_2}\dots \pi_{i_s}$ (with $i_1<i_2<\dots<i_s$) whose entries are in the same relative order as $\sigma_1\sigma_2\dots\sigma_s$.

Using the notation $[r]=\{1,2,\dots,r\}$, define $i\in[r]$ to be a {\em descent} of $\pi=\pi_1\pi_2\dots\pi_r$ if $\pi_i>\pi_{i+1}$ or $i=r$, and let 
$\des(\pi)$ denote the number of descents of $\pi$. This is the same definition used in~\cite{gessel_stirling_1978,bona_real_2008,janson_plane_2008,janson_generalized_2011}, even though other papers, such as~\cite{archer_pattern_2019}, do not consider $i=r$ to be a descent.
 Descents are closely related to {\em ascents}, which are indices $i\in\{0,\dots,r-1\}$ such that $\pi_i<\pi_{i+1}$ or $i=0$, and 
to {\em plateaus}, which are indices $i\in[r-1]$ such that $\pi_i=\pi_{i+1}$. Let $\asc(\pi)$ and $\plat(\pi)$ denote the number of ascents and the number of plateaus of $\pi$, respectively.

Denoting by $\S_n$ the set of permutations of $[n]$, the polynomials 
\begin{equation}\label{eq:Eulerian_def}
A_n(t)=\sum_{\pi\in\S_n} t^{\des(\pi)}
\end{equation}
 are called {\em Eulerian polynomials}. It is well known (see for example \cite[Prop.\ 1.4.4]{stanley_enumerative_2012}) that
\begin{equation}\label{eq:Eulerian}
\sum_{m\ge0} m^n t^m=\frac{A_n(t)}{(1-t)^{n+1}},
\end{equation}
and in fact this formula is often used as the definition of Eulerian polynomials.

Gessel and Stanley~\cite{gessel_stirling_1978} show that, when replacing the coefficients in the left-hand side of Equation~\eqref{eq:Eulerian} by Stirling numbers of the second kind, then the role of the Eulerian polynomials is played by the {\em Stirling polynomials}
$$Q_n(t)=\sum_{\pi\in\Q_n} t^{\des(\pi)},$$
which count Stirling permutations by the number of descents. Specifically, denoting by $S(n,m)$ the number of partitions of an $n$-element set into $m$ blocks, they prove the following.

\begin{theorem}[\cite{gessel_stirling_1978}]\label{thm:GS}
$$\sum_{m\ge0} S(m+n,m)\, t^m=\frac{Q_n(t)}{(1-t)^{2n+1}}.$$
\end{theorem}

It follows, in particular, that $|\Q_n|=(2n-1)!!=(2n-1)\cdot(2n-3)\cdot\dots\cdot 3\cdot 1$.

There is an extensive literature on Stirling permutations and their generalizations. B\'ona~\cite{bona_real_2008} showed that the distribution of plateaus on $\Q_n$ is also given by the polynomial $Q_n(t)$, that this polynomial has only real roots (this had also been proved by Brenti~\cite[Thm.\ 6.6.3]{brenti_unimodal_1989}), and that this distribution converges to a normal distribution. More generally, Janson~\cite{janson_plane_2008} showed that the joint distribution ascents, descents and plateaus is asymptotically normal, and Haglund and Visontai~\cite{haglund_stable_2012} proved the stability of the corresponding multivariate polynomials.

Gessel and Stanley \cite{gessel_stirling_1978} proposed an extension of Stirling permutations by allowing $k$ copies of each element in $[n]$. These permutations were studied by Brenti~\cite{brenti_unimodal_1989} in an even more general setting, proving real-rootedness of their descent polynomials; by Park~\cite{park_r-multipermutations_1994,park_inverse_1994}, who studied the distribution of various statistics on them; and by Janson, Kuba and Panholzer~\cite{janson_generalized_2011}, who proved a joint normal law for ascents, descents and plateaus. Other generalizations have been studied by Barbero et al.~\cite{barbero_g._generalized_2015}.

\subsection{Quasi-Stirling permutations}

In~\cite{archer_pattern_2019}, Archer, Gregory, Pennington and Slayden introduce the set $\QQ_n$ of {\em quasi-Stirling} permutations. These are  permutations $\pi_1\pi_2\dots\pi_{2n}$ of the multiset $\{1,1,2,2,\dots,n,n\}$ avoiding $1212$ and $2121$; that is, those for which there do not exist $i<j<k<\ell$ such that $\pi_i=\pi_k$ and $\pi_j=\pi_\ell$. Thinking of $\pi$ as a labeled matching of $[2n]$, by placing an arc between with label $k$ between $i$ with $j$ if $\pi_i=\pi_j=k$, the avoidance requirement is equivalent to the matching being {\em noncrossing} (see~\cite[Exercise 6.19(o)]{stanley_enumerative_1999}). By definition, $\Q_n\subseteq\QQ_n$.

Archer et al.~\cite{archer_pattern_2019} note that 
\begin{equation}\label{eq:nCat}
|\QQ_n|=n!\,C_n=\frac{(2n)!}{(n+1)!},\quad  \text{where }C_n=\frac{1}{n+1}\binom{2n}{n}
\end{equation} 
is the $n$th Catalan number. They also compute the number of permutations in $\QQ_n$ avoiding some sets of patterns of length 3, and they 
enumerate quasi-Stirling permutations by the number of plateaus.
They pose the open problem of enumerating quasi-Stirling permutations by the number of descents, and they conjecture the following intriguing formula\footnote{The statement of the conjecture in~\cite{archer_pattern_2019} mentions permutations with $n-1$ descents, since their definition does not consider the last position $2n$ to be a descent.}.

\begin{conjecture}[\cite{archer_pattern_2019}]\label{conj:archer}
The number of $\pi\in\QQ_n$ with $\des(\pi)=n$ is equal to $(n+1)^{n-1}$.
\end{conjecture}

\subsection{Structure of the paper}

In Section~\ref{sec:des} we prove Conjecture~\ref{conj:archer}, stated as Theorem~\ref{thm:desn}. More generally, in Theorem~\ref{thm:main}, we describe the generating function enumerating quasi-Stirling permutations by the number of descents. An analogue of Theorem~\ref{thm:GS} for quasi-Stirling permutations is given in Theorem~\ref{thm:QQn}. 
In Section~\ref{sec:properties} we study some properties of the distribution of descents on quasi-Stirling permutations, analogous to those studied by B\'ona~\cite{bona_real_2008} for Stirling permutations. We show that the corresponding polynomials have real roots only, and that the distribution of descents is asymptotically normal.

In Section~\ref{sec:k} we consider an extension of quasi-Stirling permutations by allowing $k$ copies of each element in $[n]$, in analogy to $k$-Stirling permutations. We generalize the results from Section~\ref{sec:des} to this setting, and we refine them by considering the joint distribution of the number of ascents, the number of descents, and the number of plateaus on our generalized quasi-Stirling permutations. Finally, we give a simple description of the joint distribution of the same statistics on $k$-Stirling permutations.

It is worth pointing out that Stirling permutations (and, more generally, $k$-Stirling permutations) have a simple recursive description, since elements in $\Q_n$ can be obtained by inserting the adjacent pair $nn$ into elements in $\Q_{n-1}$. This fact is repeatedly used in~\cite{gessel_stirling_1978,bona_real_2008,janson_plane_2008}, and it
greatly simplifies the enumeration of Stirling permutations by certain statistics. For example, it yields a recurrence for the number of Stirling permutations with a given number of descents or plateaus~\cite{bona_real_2008}, and it enables a probabilistic proof of the symmetry of the numbers of ascents, descents and plateaus~\cite{janson_plane_2008}. Unfortunately, quasi-Stirling permutations do not have such a simple recursive description (in fact, the quotient $|\QQ_{n}|/|\QQ_{n-1}|=2n(2n-1)/(n+1)$ is not an integer in general), which makes their enumeration with respect to the number of descents more challenging. Nevertheless, we are still able to find analogues for quasi-Stirling permutations of most of the results from~\cite{gessel_stirling_1978} and \cite{bona_real_2008}.

\subsection{A bijection to plane trees}\label{sec:bij}

The original motivation for considering quasi-Stirling permutations in~\cite{archer_pattern_2019} is that they are in bijection with labeled plane rooted trees, in much the same way that Stirling permutations are in bijection with increasing trees. Next we describe these bijections.

Denote by $\T_n$ the set of edge-labeled plane (i.e., ordered) rooted trees with $n$ edges. Each edge of such a tree receives a unique label from $[n]$. The {\em root} is a distinguished vertex of the tree, which we place at the top. The {\em children} of a vertex $i$ are the neighbors of $i$ that are not in the path from $i$ to the root; the neighbor of $i$ in the path to the root (if $i$ is not the root) is called the {\em parent} of $i$. The children of $i$ are placed below $i$, and the left-to-right order in which they are placed matters. Vertices with no children are called {\em leaves}.

Disregarding the labels, it is well known that the number of unlabeled plane rooted trees with $n$ edges is $C_n$. Since there are $n!$ ways to label the edges of a particular tree, it follows that $|\T_n|=n!\,C_n$. 

Denote by $\I_n\subseteq\T_n$ be the subset of those trees whose labels along any path from the root to a leaf are increasing. Elements of $\I_n$ are called edge-labeled increasing plane trees, or simply increasing trees when there is no confusion.

A simple bijection between $\I_n$ and $\Q_n$ was given by Janson in~\cite{janson_plane_2008}. Archer at al.~\cite{archer_pattern_2019} showed that this bijection naturally extends to a bijection $\varphi$ between $\T_n$ and $\QQ_n$. 
Both bijections can be described as follows. Given a tree $T\in\T_n$, traverse its edges by following a depth-first walk from left to right (i.e., counterclockwise); that is, start at the root, go to the leftmost child and explore that branch recursively, return to the root, then continue to the next child, and so on (see \cite[Fig.\ 5-14]{stanley_enumerative_1999} for a visual description). 
Recording the labels of the edges as they are traversed gives a permutation $\varphi(T)\in\QQ_n$;
see Figure~\ref{fig:varphi} for an example. Note that each edge is traversed twice, once in each direction.  As shown in~\cite{archer_pattern_2019}, the map $\varphi:\T_n\to\QQ_n$ is a bijection. Additionally, the image of the subset of increasing trees is precisely the set of Stirling permutations, and so $\varphi$ induces a bijection between $\I_n$ and $\Q_n$, which is the map described in~\cite{janson_plane_2008}.

\begin{figure}[htb]
\centering
 \begin{tikzpicture}[scale=1.2]
     \draw[thick] (1.5,3) coordinate(d0) -- (0,2) coordinate(d4) -- (0,1) coordinate(d1);
     \draw[thick] (d0) -- (1,2) coordinate(d6); 
     \draw[thick] (d0) -- (2.5,2) coordinate(d3) -- (1.5,1) coordinate(d7);
     \draw[thick] (d3) -- (2.5,1) coordinate(d5) -- (2.5,0) coordinate(d8);
     \draw[thick] (d3) -- (3.5,1) coordinate(d2);
      \foreach \x in {0,...,8} {
        \draw[fill] (d\x) circle (2pt);
      }
      \foreach \x in {4,6} {
        \draw (d\x)+(.45,.5) node {$\x$};
      }
      \foreach \x in {7} {
        \draw (d\x)+(.2,.5) node {$\x$};
      }
      \foreach \x in {3,2,5,8,1} {
        \draw (d\x)+(-.15,.5) node {$\x$};
      }
      \draw (4,1.5) node[right] {$\longrightarrow \quad 4114663775885223$};
      \draw (4.15,1.5) node[above right]{$\varphi$};
      \end{tikzpicture}
      \caption{An example of the bijection $\varphi:\T_n\to\QQ_n$.}
      \label{fig:varphi}
\end{figure}
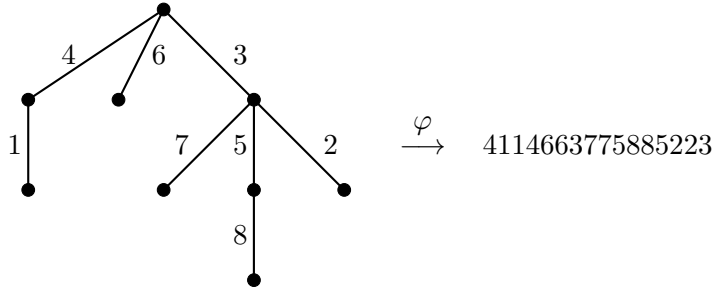

\section{Descents on quasi-Stirling permutations}\label{sec:des}

In order to enumerate quasi-Stirling permutations by the number of descents, let us first analyze how descents are transformed by the bijection $\varphi$.

Define the number of {\em cyclic descents} of a sequence of positive integers $\pi=\pi_1\pi_2\dots\pi_r$ to be
\begin{equation}\label{eq:cdes_def} \cdes(\pi)=|\{i\in[r]: \pi_i>\pi_{i+1}\}|,
\end{equation}
with the convention $\pi_{r+1}:=\pi_1$. Note that rotating the entries of $\pi$ does not change the number of cyclic descents, that is,
$\cdes(\pi_{i+1}\dots\pi_r\pi_1\dots\pi_i)=\cdes(\pi)$ for all $i\in[r]$.

Let $T\in\T_n$, and let $v$ a vertex of $T$. If $v$ is not the root, define $\cdes(v)$ to be the number of cyclic descents of the sequence obtained by listing the labels of the edges incident to $v$ in counterclockwise order (note that the starting point is irrelevant). Equivalently, if the
label of the edge between $v$ and its parent is $\ell$, and the labels of the edges between $v$ and its children are $a_1,a_2,\dots,a_{d}$ from left to right, then $\cdes(v)=\cdes(\ell a_1\dots a_d)$. If $v$ is the root of $T$, define $\cdes(v)$ to be the number of descents of the sequence 
obtained by listing the labels of the edges incident to $v$ from left to right, that is, $\cdes(v)=\des(a_1\dots a_d)$ with the above notation. Finally, define the number of cyclic descents of $T$ to be $$\cdes(T)=\sum_v \cdes(v),$$ where the sum ranges over all the vertices $v$ of $T$.
For example, if $T$ is the tree on the left of Figure~\ref{fig:varphi}, then $\cdes(T)=2+1+0+0+2+0+1+0+0=6$, where the two vertices with $\cdes(v)=2$ are the root and the other vertex with 3 children.

\begin{lemma}\label{lem:cdes}
The bijection $\varphi:\T_n\to\QQ_n$ has the following property: if $T\in\T_n$ and $\pi=\varphi(T)\in\QQ_n$, then $$\des(\pi)=\cdes(T).$$
\end{lemma}

\begin{proof}
In the counterclockwise depth-first walk performed on $T$ to obtain $\pi$, 
suppose that the $i$th step of the walk traverses edge $e$ in the direction from vertex $u$ to vertex $v$, and let $\ell$ be its label. Then $i$ is a descent of $\pi$ 
if and only if one of the following holds:
\begin{itemize}
\item $u$ is the parent of $v$, and $\ell$ is larger than the label of the edge between $v$ and its leftmost child;
\item $u$ is a child of $v$ (but not its rightmost child), and $\ell$ is larger than the label of the edge between $v$ and its next child in the left-right order;
\item $u$ is the rightmost child of $v$, and either $v$ is the root or $T$, or $\ell$ is larger than the label of the edge between $v$ and its parent.
\end{itemize}
By definition, $\cdes(v)$ counts the number of times that $v$ is involved in one of these scenarios at some point in the depth-first walk on $T$, and so $\cdes(T)$ equals the total number of descents of~$\pi$.
\end{proof}

Another property of $\varphi$ that follows easily from its description, as noted in~\cite{archer_pattern_2019}, is that plateaus of quasi-Stirling permutations correspond to leaves of edge-labeled plane rooted trees. Specifically, denoting by $\lea(T)$ the number of leaves of $T\in\T_n$, and letting $\pi=\varphi(T)\in\QQ_n$, we have that
\begin{equation}\label{eq:lea}
\plat(\pi)=\lea(T).
\end{equation}

\subsection{Quasi-Stirling permutations with most descents}

It follows from Lemma~\ref{lem:cdes} that the maximum value that $\des(\pi)$ can attain for $\pi\in\QQ_n$ is~$n$. To see this, note that if $\pi=\varphi(T)$ and $v$ is a vertex of $T$, then $\cdes(v)$ is bounded from above by the number of children of $v$, which we denote by $\d(v)$. 
Summing over all the vertices $v$ of $T$, we get $\des(\pi)=\cdes(T)=\sum_{v} \cdes(v) \le \sum_{v} d(v)=n$, the number of edges of $T$. This bound is attained, for example, by the permutation $\pi=12\dots nn\dots 21\in\QQ_n$.
Next we count how many permutations attain this upper bound, proving Conjecture~\ref{conj:archer}.

\begin{theorem}\label{thm:desn}
The number of $\pi\in\QQ_n$ with $\des(\pi)=n$ is equal to $(n+1)^{n-1}$.
\end{theorem}

\begin{proof}
By Lemma~\ref{lem:cdes}, the problem is equivalent to counting the number of trees $T\in\T_n$ such that $\cdes(T)=n$.
As discussed above, $\cdes(T)=n$ if and only if $\cdes(v)=\d(v)$ for every vertex $v$ of $T$. 
Let $\T^{\max}_n\subseteq\T_n$ be the subset of trees satisfying this condition. 
If $v$ is the root of $T\in\T_n$, then $\cdes(v)=\d(v)$ precisely when the labels of the edges incident to $v$ decrease from left to right. If $v$ is a non-root vertex, then $\cdes(v)=\d(v)$ if and only if the labels of the edges incident to $v$ decrease when read counterclockwise starting from the largest label. 

Let $\U_n$ be the set of edge-labeled {\em unordered} rooted trees with $n$ edges. The difference with $\T_n$ is that, for trees in $\U_n$, the order of the children of a vertex is irrelevant; that is, trees are determined by their combinatorial structure and not by their particular embedding on the plane. It is known that $|\U_n|=(n+1)^{n-1}$. Indeed, transferring edge labels to vertex labels by moving each label to the endpoint away from the root, and labeling the root with $n+1$, trees in $\U_n$ are in bijection with vertex-labeled unordered unrooted trees on $n+1$ vertices (an example of this bijection appears on the left of Figure~\ref{fig:unordered}). By 
Cayley's formula, the number of such trees is $(n+1)^{n-1}$. Thus, it suffices to describe a bijection between $\T^{\max}_n$ and $\U_n$. 

Given a tree in $\T^{\max}_n$, the corresponding unordered tree is obtained simply by forgetting the order of the children of each vertex. Conversely, given a tree in $\U_n$, the unique tree in $\T^{\max}_n$ with the same combinatorial structure is obtained as follows. First, place the root and its children so that the labels of the corresponding edges decrease from left to right. Then, for each placed vertex $v$, recursively place its children in the only possible order that yields $\cdes(v)=\d(v)$. Specifically, writing $d=d(v)$, suppose that the labels between $v$ and its children are $a_1,a_2,\dots,a_d$ in increasing order, that the label between $v$ and its parent is $a$, and that $0\le i\le d$ is the index such that 
$$a_1<a_2<\dots<a_i<a<a_{i+1}<\dots<a_d.$$
Then place the children of $v$ so that the edge labels are $a_i,a_{i-1},\dots,a_1,a_d,a_{d-1},\dots,a_{i+1}$ from left to right.
This guarantees that $\cdes(aa_ia_{i-1}\dots a_1a_da_{d-1}\dots a_{i+1})=d$. An example of this bijection is shown on the right of Figure~\ref{fig:unordered}.
\end{proof}

\begin{figure}[htb]
\centering
\begin{tikzpicture}[scale=1.1]
     \draw[thick] (0,0) coordinate(d1) -- (1,0) coordinate(d4) -- (2,0) coordinate(d9) -- (3,0) coordinate(d3) -- (4,0) coordinate(d5) -- (5,0) coordinate(d8);
     \draw[thick] (d9) -- (2,-1) coordinate(d6); 
     \draw[thick] (3,1) coordinate(d2) -- (d3) -- (3,-1) coordinate(d7);
      \foreach \x in {1,...,9} {
        \draw[fill] (d\x) circle (2pt);
      }
      \foreach \x in {1,4,9,5,8} {
        \draw (d\x) node[above] {$\x$};
      }
      \foreach \x in {2,6,7} {
        \draw (d\x) node[right] {$\x$};
      }
      \foreach \x in {3} {
        \draw (d\x) node[above right] {$\x$};
      }
      \end{tikzpicture}
      \quad
 \begin{tikzpicture}[scale=1.1]
     \draw[thick] (1.5,3) coordinate(d0) -- (0,2) coordinate(d4) -- (0,1) coordinate(d1);
     \draw[thick] (d0) -- (1,2) coordinate(d6); 
     \draw[thick] (d0) -- (2.5,2) coordinate(d3) -- (1.5,1) coordinate(d7);
     \draw[thick] (d3) -- (2.5,1) coordinate(d5) -- (2.5,0) coordinate(d8);
     \draw[thick] (d3) -- (3.5,1) coordinate(d2);
      \foreach \x in {0,...,8} {
        \draw[fill] (d\x) circle (2pt);
      }
      \foreach \x in {4,6} {
        \draw (d\x)+(.45,.5) node {$\x$};
      }
      \foreach \x in {7} {
        \draw (d\x)+(.2,.5) node {$\x$};
      }
      \foreach \x in {3,2,5,8,1} {
        \draw (d\x)+(-.15,.5) node {$\x$};
      }
      \draw (1.5,3.5) node{$\U_n$};
      \end{tikzpicture}
      \quad
      \begin{tikzpicture}[scale=1.1]
     \draw[thick] (1.5,3) coordinate(d0) -- (1,2) coordinate(d4) -- (1,1) coordinate(d1);
     \draw[thick] (d0) -- (0,2) coordinate(d6); 
     \draw[thick] (d0) -- (2.5,2) coordinate(d3) -- (1.5,1) coordinate(d2);
     \draw[thick] (d3) -- (3.5,1) coordinate(d5) -- (3.5,0) coordinate(d8);
     \draw[thick] (d3) -- (2.5,1) coordinate(d7);
      \foreach \x in {0,...,8} {
        \draw[fill] (d\x) circle (2pt);
      }
      \foreach \x in {4,6} {
        \draw (d\x)+(.45,.5) node {$\x$};
      }
      \foreach \x in {2} {
        \draw (d\x)+(.2,.5) node {$\x$};
      }
      \foreach \x in {3,7,5,8,1} {
        \draw (d\x)+(-.15,.5) node {$\x$};
      }
      \draw (1.5,3.5) node{$\T^{\max}_n$};
      \end{tikzpicture}
      \caption{The bijections in the proof of Theorem~\ref{thm:desn}: a vertex-labeled unordered unrooted tree (left), its corresponding 
edge-labeled unordered rooted tree (center), and its corresponding 
edge-labeled plane rooted tree with maximum number of descents (right).}
      \label{fig:unordered}
\end{figure}
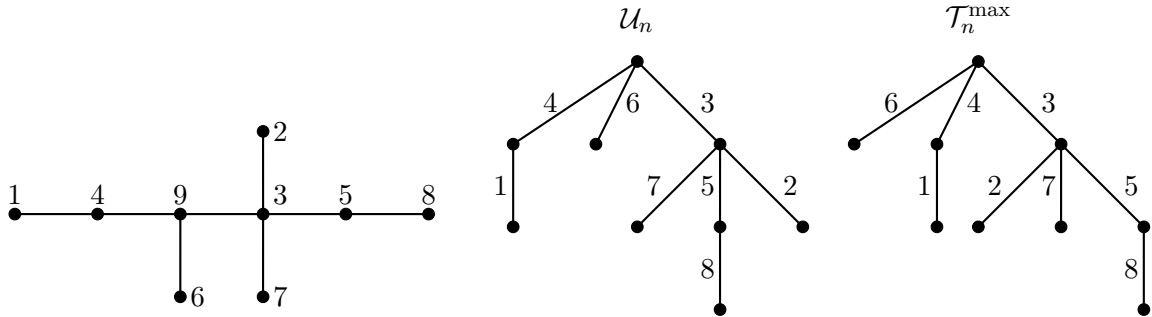

\subsection{A generating function for the number of descents}

Denote by $$A(t,z)=\sum_{n\ge0} A_n(t) \frac{z^n}{n!}$$
the exponential generating function (EGF for short) of the  Eulerian polynomials, defined in Equation~\eqref{eq:Eulerian}. It is well known \cite[Prop.\ 1.4.5]{stanley_enumerative_2012} that
\begin{equation}\label{eq:A_GF} A(t,z)=\frac{1-t}{1-te^{(1-t)z}}.
\end{equation}

In analogy to $A_n(t)$ and $Q_n(t)$, define the {\em quasi-Stirling polynomials}
\begin{equation}\label{eq:defoQn}
\oQ_n(t)=\sum_{\pi\in\QQ_n} t^{\des(\pi)},
\end{equation} 
and their EGF
$$\oQ(t,z)=\sum_{n\ge0}\oQ_n(t)\frac{z^n}{n!}.$$
The first few quasi-Stirling polynomials are
\begin{align*}
&\oQ_1(t)=t,\\ 
&\oQ_2(t)=t+3\,{t}^{2},\\
&\oQ_3(t)=t+13\,{t}^{2}+16\,{t}^{3}, \\
&\oQ_4(t)=t+39\,{t}^{2}+171\,{t}^{3}+125\,{t}^{4},\\
&\oQ_5(t)=t+101\,{t}^{2}+1091\,{t}^{3}+2551\,{t}^{4}+1296\,{t}^{5},\\
&\oQ_6(t)=t+243\,{t}^{2}+5498\,{t}^{3}+28838\,{t}^{4}+43653\,{t}^{5}+16807\,{t}^{6},\\
&\oQ_7(t)=t+561\,{t}^{2}+24270\,{t}^{3}+243790\,{t}^{4}+780585\,{t}^{5}+850809\,{t}^{6}+262144\,{t}^{7}.
\end{align*}

The main result in this section is an equation that describes $\overline{Q}(t,z)$, and allows us to compute the quasi-Stirling polynomials. The notation $[z^n]F(z)$ refers to the coefficient of $z^n$ in the generating function $F(z)$.

\begin{theorem}\label{thm:main}
The EGF of quasi-Stirling permutations by the number of descents satisfies the implicit equation $\oQ(t,z)=A(t,z\oQ(t,z))$, that is,
\begin{equation}\label{eq:oQ}
\oQ(t,z)=\frac{1-t}{1-te^{(1-t)z\oQ(t,z)}}.
\end{equation}
In particular, its coefficients satisfy
\begin{equation}\label{eq:Lagrange}
\oQ_n(t)=\frac{n!}{n+1}\,[z^n]A(t,z)^{n+1}.
\end{equation}
\end{theorem}

Before proving this theorem, note that if we ignore descents by setting $t=1$, then $A(1,t)=\frac{1}{1-z}$. In this case, Theorem~\ref{thm:main} simply states that 
$$\oQ(1,z)=A(1,z\oQ(t,z))=\frac{1}{1-z\oQ(1,z)}.$$ It follows that 
\begin{equation}\label{eq:Catalan}
\oQ(1,z)=\frac{1-\sqrt{1-4z}}{2z},
\end{equation} 
the ordinary generating function for the Catalan numbers, and so $|\QQ_n|=n!C_n$, which agrees with Equation~\eqref{eq:nCat}.
Equation~\eqref{eq:Lagrange} in this case states that 
$$\oQ_n(1)=\frac{n!}{n+1}\,[x^{n}]\frac{1}{(1-x)^{n+1}}=\frac{n!}{n+1}\binom{2n}{n}=n!C_n.$$

It is also worth noting that a non-bijective proof of Theorem~\ref{thm:desn} can be deduced from Theorem~\ref{thm:main}, by noting that the number of $\pi\in\QQ_n$ with $n$ descents is
$$[t^n]\oQ_n(t)=\frac{n!}{n+1}\,[t^nz^n]A(t,z)^{n+1}=\frac{n!}{n+1}\,[t^nz^n]\left(\sum_{n\ge0}\frac{t^nz^n}{n!}\right)^{n+1}=\frac{n!}{n+1}\,[t^nz^n]e^{(n+1)tz}=(n+1)^{n-1}.$$

By Lemma~\ref{lem:cdes}, $\oQ(t,z)$ is also the EGF for edge-labeled plane rooted trees by the number of cyclic descents, that is, 
\begin{equation}\label{eq:oQT}
\oQ_n(t)=\sum_{T\in\T_n} t^{\cdes(T)} \quad \text{and}\quad \oQ(t,z)=\sum_{n\ge0} \sum_{T\in\T_n} t^{\cdes(T)} \frac{z^n}{n!}.
\end{equation}
For $n\ge1$, let $\T'_n\subseteq\T_n$ be the subset of trees whose root has exactly one child, and let 
\begin{equation}
\label{eq:Rn}
R_n(t)=\sum_{T'\in\T'_n} t^{\cdes(T')-1}.
\end{equation}
The next lemma will be used in the proof of Theorem~\ref{thm:main}. 

\begin{lemma}\label{lem:T'}
We have
$$\sum_{n\ge1} R_n(t) \frac{z^n}{n!}=z\,\oQ(t,z).$$
\end{lemma}

\begin{proof}
We define two operations on trees. For $T\in\T_{n-1}$, let $T^{|n}\in\T'_n$ be the tree obtained from $T$ by attaching an edge with label $n$ from the root of $T$ to a new root vertex. For $T'\in\T'_n$ and $i\in\{0,1,\dots,n-1\}$, let $T'+i\in\T'_n$ be tree obtained by adding $i$ modulo $n$ to each label of $T'$, so that the resulting labels are again the numbers $1,2,\dots,n$.

Next we analyze how the statistic $\cdes$ behaves under these two operations. We have that $\cdes(T^{|n})=\cdes(T)+1$ for all $T\in\T_{n-1}$, since $\cdes(v)$ stays the same for each vertex $v$ of $T$, and the new root contributes one cyclic descent. On the other hand, $\cdes(T'+i)=\cdes(T')$ for all $T'\in\T'_n$ because, when adding $1$ modulo $n$ to each label, the relative order of the labels around a vertex $v$ does not change unless $v$ is an endpoint of the edge with label $n$ in $T'$, which has label $1$ in $T'+1$. But, even for such $v$, the value of $\cdes(v)$ does not change, because when reading the labels around $v$ in counterclockwise order, the old label $n$ was bigger than the next label, whereas the new label $1$ is smaller than the previous label.

Since every tree in $T'\in\T'_n$ can be obtained as $T'=T^{|n}+i$ for a unique $i\in\{0,1,\dots,n-1\}$ and a unique $T\in\T_{n-1}$, we have
\begin{align*} R_n(t)&=\sum_{T'\in\T'_n} t^{\cdes(T')-1}=\sum_{T\in\T_{n-1}} \sum_{i=0}^{n-1} t^{\cdes(T^{|n}+i)-1}=\sum_{T\in\T_{n-1}} n t^{\cdes(T^{|n})-1}
=n \sum_{T\in\T_{n-1}} t^{\cdes(T)}\\
&=n\oQ_{n-1}(t)
\end{align*}
for every $n\ge1$, where we used Equation~\eqref{eq:oQT} in the last equality. Multiplying both sides by $z^n/n!$ and summing over $n$, we get
$$\sum_{n\ge1} R_n(t) \frac{z^n}{n!}=\sum_{n\ge1} \oQ_{n-1}(t) \frac{z^n}{(n-1)!}=z\,\oQ(t,z).$$
\end{proof}

\begin{proof}[Proof of Theorem~\ref{thm:main}]
We consider a recursive description of edge-labeled plane rooted trees.
Let $\T=\bigcup_{n\ge0}\T_n$ and $\T'=\bigcup_{n\ge1}\T'_n$. We think of these sets as labeled combinatorial classes in the sense of Flajolet and Sedgewick~\cite{flajolet_analytic_2009}. The EGFs for these classes with a variable $t$ keeping track of the number of cyclic descents are $\oQ(t,z)$ for $\T$, by Equation~\eqref{eq:oQT}, and $tz\oQ(t,z)$ for $\T'$, by Lemma~\ref{lem:T'}. 

Every tree in $\T$ can be decomposed as a sequence of trees in $\T'$ with a common root. Let us first consider the simplified version where we
momentarily disregard  the parameter $\cdes$. Recall from~\cite{flajolet_analytic_2009} that if $F(z)$ is the EGF for a labeled class, then the EGF for sequences of objects in that class is $\frac{1}{1-F(z)}$. Thus, the decomposition of trees in $\T$ as sequences of trees in $\T'$ yields the equation
$$\oQ(1,z)=\frac{1}{1-z\oQ(1,z)}.$$

Next we incorporate the parameter $\cdes$, by analyzing how it behaves under this decomposition. 
Suppose that $T\in\T$ is obtained by attaching $r$ non-empty trees $T'_1,T'_2,\dots,T'_r\in\T'$ to a common root, denoted by $v_0$, and relabeling their edges with distinct labels from $1$ up to the total number of edges, so that the relative order of the labels within each tree is preserved. This relabeling does not change the value of $\cdes(v)$ at any vertex $v$ other than $v_0$. However, whereas the root of each $T'_i$ contributed $1$ to its number of cyclic descents, the contribution to $\cdes(T)$ of the common root $v_0$ equals the number of descents of the sequence of labels in $T$ of the edges incident to this vertex. Specifically, if the edges incident to $v_0$ in $T$ have labels $a_1,a_2,\dots,a_r$ from left to right, then
$$\cdes(T)=\sum_{i=1}^r(\cdes(T'_i)-1)+\des(a_1a_2\dots a_r).$$

Since the labels of these $r$ subtrees of $T$ form a partition of $[n]$, and the order in which these trees are attached to $v_0$ can be any of the $r!$ permutations, it follows that
$$\sum_{T\in\T_n} t^{\cdes(T)} = \sum_{\substack{\{B_1,B_2,\dots,B_r\}\\ \text{partition of $[n]$}}} \left(\sum_{T'_1\in\T'_{|B_1|}} t^{\cdes(T'_1)-1}\right)\cdots\left(\sum_{T'_r\in\T'_{|B_r|}} t^{\cdes(T'_r)-1}\right)\left(\sum_{\pi\in\S_r} t^{\des(\pi)}\right),$$
where the first sum on the right-hand side is over all partitions of $[n]$, not just those with a fixed number of blocks.
Using the notation $R_n(t)$ from Equation~\eqref{eq:Rn}, we can rewrite this equation as
\begin{align*}\oQ_n(t)&=\sum_{\substack{\{B_1,B_2,\dots,B_r\}\\ \text{partition of $[n]$}}} R_{|B_1|}(t)  \cdots R_{|B_r|}(t) A_r(t)\\
&=\sum_{r=1}^n\frac{1}{r!}\sum_{\substack{b_1+\dots+b_r=n \\ b_1,\dots,b_r\ge1}}\binom{n}{b_1,\dots,b_r} R_{b_1}(t)  \cdots R_{b_r}(t) A_r(t).
\end{align*}
To turn this equality into an equation for EGFs, we use a bivariate version of the Compositional Formula (see \cite[Thm.\ 5.1.4]{stanley_enumerative_1999}). 
Explicitly, we multiply both sides by $z^n/n!$, sum over $n\ge0$, and apply Lemma~\ref{lem:T'}:
\begin{align*}\oQ(t,z)=\sum_{n\ge0}\oQ_n(t)\frac{z^n}{n!}&=1+\sum_{n\ge1}\sum_{r=1}^n\frac{1}{r!}\sum_{\substack{b_1+\dots+b_r=n\\ b_1,\dots,b_r\ge1}}
\frac{1}{b_1!\cdots b_r!} R_{b_1}(t) \cdots R_{b_r}(t) A_r(t) z^n\\
&=1+\sum_{r\ge1}\frac{1}{r!}\sum_{b_1,\dots,b_r\ge1} R_{b_1}(t)\frac{z^{b_1}}{b_1!}  \cdots R_{b_r}(t)\frac{z^{b_r}}{b_r!} A_r(t)\\
&=1+\sum_{r\ge1} \frac{1}{r!}(z\oQ(t,z))^r A_r(t) \\
&=A(t,z\oQ(t,z)),
\end{align*}
proving Equation~\eqref{eq:oQ}.

Next, we extract the coefficients of $\oQ(t,z)$. The generating function $F(t,z):=z\oQ(t,z)$ satisfies the equation
$$F(t,z)=zA(t,F(t,z)).$$
Thus, by the Lagrange's inversion formula (see for example \cite[Thm.\ 5.4.2]{stanley_enumerative_1999}), we have
$$[z^n]F(t,z)=\frac{1}{n}\,[z^{n-1}]A(t,z)^n,$$
so
$$\oQ_n(t)=n!\,[z^n]\oQ(t,z)=n!\,[z^{n+1}]F(t,z)=\frac{n!}{n+1}\,[z^{n}]A(t,z)^{n+1}.$$
\end{proof}

Gessel and Stanley's main result from~\cite{gessel_stirling_1978} (stated above as Theorem~\ref{thm:GS}) is the analogue for Stirling polynomials 
of Equation~\eqref{eq:Eulerian} for Eulerian polynomials.
As a consequence of Theorem~\ref{thm:main}, we obtain the following analogue for quasi-Stirling polynomials of these two results.

\begin{theorem}\label{thm:QQn}
$$\sum_{m\ge 0} \frac{m^n}{n+1}\binom{m+n}{m}\, t^m=\frac{\oQ_n(t)}{(1-t)^{2n+1}}.$$
\end{theorem}

\begin{proof}
We use Equation~\eqref{eq:Lagrange} and extract the coefficient of $z^n$ in $A(t,z)^{n+1}$. By Equation~\eqref{eq:A_GF}, this expression equals
$$\left(\frac{1}{1-te^{(1-t)z}}\right)^{n+1}=\sum_{m\ge0} \binom{m+n}{m}t^m e^{m(1-t)z},$$
and so
$$[z^n]\left(\frac{1}{1-te^{(1-t)z}}\right)^{n+1}=\sum_{m\ge0} \binom{m+n}{m} \frac{t^m m^n(1-t)^n}{n!}.$$
Thus, by Equation~\eqref{eq:Lagrange},
$$\oQ_n(t)=\frac{n!}{n+1}[z^n]\left(\frac{1-t}{1-te^{(1-t)z}}\right)^{n+1}=\frac{(1-t)^{2n+1}}{n+1}\sum_{m\ge0} \binom{m+n}{m} m^n t^m,$$
which is equivalent to the stated formula.
\end{proof}

\section{Properties of quasi-Stirling polynomials}\label{sec:properties}

In the section, we prove some properties of the distribution of the number of descents ---as well as the number of ascents and the number of plateaus--- on quasi-Stirling permutations, in analogy with B\'ona's results for Stirling permutations~\cite{bona_real_2008}.

By symmetry, the number of ascents and the number of descents are equidistributed on $\QQ_n$, since reversing a permutation in $\QQ_n$ turns ascents into descents and vice versa. 

Using Equation~\eqref{eq:lea}, it is shown by Archer et al.~\cite{archer_pattern_2019} that the number of elements in $\QQ_n$ with $m$ plateaus is
$$n!N(n,m)=(n-1)!\binom{n}{m}\binom{n}{m-1},$$
where the Narayana number $N(n,m)$ is the number of unlabeled plane rooted trees with $n$ edges and $m$ leaves. The corresponding generating function is
\begin{equation}\label{eq:Narayana}
\sum_{n\ge0} \sum_{\pi\in\QQ_n} u^{\plat(\pi)} \frac{z^n}{n!}
=1+\sum_{n,m\ge1} N(n,m) u^m z^n=\frac{1-(u-1)z-\sqrt{1-2(1+u)z+(1-u)^2z^2}}{2z}.
\end{equation}

\subsection{Average number of ascents, descents, and plateaus}

B\'ona proves in~\cite[Cor.\ 1]{bona_real_2008} that Stirling permutations in $\Q_n$ have, on average, $(2n+1)/3$ ascents, $(2n+1)/3$ descents, and $(2n+1)/3$ plateaus. From Theorem~\ref{thm:main}, we can derive the following analogue for quasi-Stirling permutations.

\begin{corollary}\label{cor:average}
Let $n\ge1$. On average, elements of $\QQ_n$ have $(3n+1)/4$ ascents, $(3n+1)/4$ descents, and $(n+1)/2$ plateaus.
\end{corollary}
 
\begin{proof}
By Equation~\eqref{eq:defoQn}, 
$$\frac{\oQ_n'(1)}{|\QQ_n|}=\frac{\sum_{\pi\in\QQ_n} \des(\pi)}{|\QQ_n|}$$
is the average number of descents in elements of $\QQ_n$.

Differentiating Equation~\eqref{eq:oQ} with respect to $t$, setting $t=1$, and solving for $\frac{\partial \oQ}{\partial t}(1,z)$, we get
\begin{equation}\label{eq:oQ'}
\frac{\partial \oQ}{\partial t}(1,z)=\frac{z\oQ(1,z)(2-z\oQ(1,z))}{2(1-z-2z\oQ(1,z)+z^2\oQ(1,z)^2)}=\frac{1}{4z}\left(\frac{1-z}{\sqrt{1-4z}}-1+z\right),
\end{equation}
where we have used Equation~\eqref{eq:Catalan} in the last equality.

Extracting the coefficient of $z^n$ for $n\ge1$ on both sides of~\eqref{eq:oQ'},
$$\frac{\oQ_n'(1)}{n!}=\frac{1}{4}\,[z^{n+1}]\frac{1-z}{\sqrt{1-4z}}=\frac{1}{4}\left(\binom{2n+2}{n+1}-\binom{2n}{n}\right)=\frac{3n+1}{4(n+1)}\binom{2n}{n}.$$
Dividing by $|\QQ_n|/n!=C_n$, we conclude that
$$\frac{\oQ_n'(1)}{|\QQ_n|}=\frac{3n+1}{4}.$$

The average number of descents in elements of $\QQ_n$ equals the average number of ascents, by symmetry. Additionally, the sum of the numbers of ascents, descents and plateaus of any given $\pi\in\QQ_n$ is $2n+1$, since every $i\in\{0,1,\dots,2n\}$ is an ascent, a descent or a plateau of $\pi$.  It follows that the average number of plateaus is
$$2n+1-2\cdot\frac{3n+1}{4}=\frac{n+1}{2}.$$
Alternatively, this average can be deduced directly from the generating function in Equation~\eqref{eq:Narayana}.
\end{proof}

\subsection{Real roots of quasi-Stirling polynomials and $r$-Eulerian polynomials}

It is well-known result of Frobenius that the roots of the Eulerian polynomials $A_n(t)$ are real, distinct, and nonpositive. In~\cite[Thm.\ 1]{bona_real_2008}, B\'ona proves the analogous result for the Stirling polynomials $Q_n(t)$, although their real-rootedness had already been shown by Brenti~\cite[Thm.\ 6.6.3]{brenti_unimodal_1989} in more generality. In this subsection, we prove that quasi-Stirling polynomials $\oQ_n(t)$ also have this property. Unlike the proofs for $A_n(t)$ and $Q_n(t)$ that give direct recurrences for these polynomials, our proof relates $\oQ_n(t)$ to the so-called $r$-Eulerian polynomials.

For $r\ge1$, define the number of $r$-excedances of a sequence $\pi=\pi_1\pi_2\dots\pi_s$ to be
$$\exc_r(\pi)=\{i\in[s]:\pi_i\ge i+r\}.$$
In particular, we write $\exc(\pi)=\exc_1(\pi)$ to denote the number of excedances in the usual sense.

Riordan~\cite{riordan_introduction_1958}, and later Foata and Sch\"utzenberger~\cite{foata_theorie_1970}, defined the polynomials 
$$A_{n,r}(t)=\sum_{\pi\in\S_n} t^{\exc_r(\pi)}.$$
For $r=1$, we have $A_{n,1}(t)=A_n(t)/t$, by the well-known fact 
(see \cite{foata_theorie_1970} or \cite[Prop.\ 1.4.3]{stanley_enumerative_2012}) that the number of excedances in $\S_n$ is equidistributed with the number of descents, if we do not consider the last position $n$ to be a descent.
Let $\In_{n,r}$ denote the set of injections $\pi:[n-r]\to[n]$. Identifying such an injection with the sequence $\pi=\pi_1\pi_2\dots\pi_{n-r}$ of its images, define the polynomials
$$J_{n,r}(t)=\sum_{\pi\in\In_{n,r}} t^{\exc(\pi)}.$$
For small values of $r$, these polynomials appear in \cite[A144696--A144699]{sloane_-line_nodate}. Adapting the notation, it is shown in~\cite{riordan_introduction_1958,foata_theorie_1970} that, for $r\ge1$,
\begin{equation}\label{eq:IA}
J_{n,r}(t)=\frac{t^{n-r}\,A_{n,r}(1/t)}{r!}
\end{equation}
(in particular, $J_{n,1}(t)=A_{n,1}(t)=A_n(t)/t$) and that 
$$\sum_{m\ge1} t\,J_{m+r-1,r}(t)\,\frac{z^m}{m!}=\frac{A(t,z)^r}{r}.$$

Setting $r=n+1$ and taking the coefficient of $z^n$, it follows from Theorem~\ref{thm:main} that
\begin{equation}\label{eq:QI}
\oQ_n(t)=t\,J_{2n,n+1}(t).
\end{equation}

We are now ready to prove the real-rootedness of quasi-Stirling and $r$-Eulerian polynomials.

\begin{theorem}\label{thm:roots}
For every $1\le r\le n$, each of the polynomials $A_{n,r}(t)$, $J_{n,r}(t)$ and $\oQ_n(t)$ has real, distinct, and nonpositive roots.
\end{theorem}

\begin{proof}
We will prove that the polynomials
\begin{equation}\label{eq:def_pnr}
p_{n,r}(t):=t\,J_{n,r}(t)
\end{equation}
have real, distinct, and nonpositive roots.
The statement for $J_{n,r}(t)$ will then follow immediately, as well as for $A_{n,r}(t)$ because, by Equation~\eqref{eq:IA}, its roots are the reciprocals of the roots of $J_{n,r}(t)$. The statement for the polynomials $\oQ_n(t)$ will follow from Equation~\eqref{eq:QI}.

First we claim that, for any fixed $r\ge1$, the polynomials $p_{n,r}(t)$ satisfy the recurrence
\begin{equation}\label{eq:recp}
p_{n,r}(t)=n\,t\,p_{{n-1},r}(t)+t(1-t)\,p_{n-1,r}'(t)
\end{equation}
for $n>r$, with initial condition $p_{r,r}(t)=t$.
Indeed, this is a direct translation, using Equations~\eqref{eq:IA} and~\eqref{eq:def_pnr}, of the recurrence for $A_{n,r}(t)$ proved in \cite[p.\ 214]{riordan_introduction_1958}:
$$A_{n,r}(t)=(r+(n-r)t)A_{n-1,r}(t)+t(1-t)A'_{n-1,r}(t).$$
Note that, aside from the initial condition, recurrence~\eqref{eq:recp} does not depend on $r$, and it extends the well-known recurrence satisfied by the Eulerian polynomials $A_n(t)=p_{n,1}(t)$.

By definition, $p_{n,r}(t)$ is a polynomial of degree $n-r+1$ with a positive leading coefficient, and $0$ is one of its roots.
Next we show by induction on $n$ that $p_{n,r}(t)$ has $n-r+1$ real, distinct, and nonpositive roots. This is trivially true for base case $n=r$, since $p_{r,r}(t)=t$. 

Let $n>r$, and suppose that $p_{n-1,r}(t)$ has $n-r$ real, distinct roots $x_1<x_2<\dots<x_{n-r}=0$. 
The sign of the derivative $p'_{n-1,r}(t)$ alternates on these roots; specifically, $p'_{n-1,r}(x_i)$ is positive if $i$ and $n-r$ have the same parity, and negative otherwise. Since $p_{{n-1},r}(x_i)=0$ and $1-x_i>0$ for all $i$, the same assertion applies to the sign of
$n\,p_{{n-1},r}(x_i)+(1-x_i)\,p_{n-1,r}'(x_i)$.

It follows that the polynomial $n\,p_{{n-1},r}(t)+(1-t)\,p_{n-1,r}'(t)$ must have a root between any pair of consecutive roots of $p_{n-1,r}(t)$, let us denote these roots by $y_1,\dots,y_{n-r-1}$ where $$x_1<y_1<x_2<y_2<\dots<y_{n-r-1}<x_{n-r}=0.$$
Using that 
$$p_{n,r}(t)=t\left(n\,p_{{n-1},r}(t)+(1-t)\,p_{n-1,r}'(t)\right)$$
by Equation~\eqref{eq:recp}, the polynomial $p_{n,r}(t)$ has the roots $y_1,\dots,y_{n-r-1}$, plus a root $y_{n-r}=0$.
Note also that, if $n-r$ is even, then $p_{n,r}(x_1)>0$ and $\lim_{n\to-\infty}p_{n,r}(t)=-\infty$; if $n-r$ is odd, then $p_{n,r}(x_1)<0$ and $\lim_{n\to-\infty}p_{n,r}(t)=+\infty$. In both cases, $p_{n,r}(t)$ has an additional root $y_0<x_1$, for a total of $n-r+1$ roots $y_0<y_1<\dots<y_{n-r-1}<y_{n-r}=0$.
\end{proof}

In analogy to B\'ona's results for Stirling permutations \cite[Thm.\ 3]{bona_real_2008}, we can infer the modal number of descents in $\QQ_n$
from Theorem~\ref{thm:roots}.

\begin{corollary}
Fix $n\ge1$, and let $m$ be an index that maximizes $|\{\pi\in\QQ_n:\des(\pi)=m\}|$ (equivalently, $|\{\pi\in\QQ_n:\asc(\pi)=m\}|$). Then 
$$\left|m-\frac{3n+1}{4}\right|<1.$$ 
Similarly, let $m'$ be an index that maximizes $|\{\pi\in\QQ_n: \plat(\pi)=m'\}|$. Then 
$$\left|m'-\frac{n+1}{2}\right|<1.$$ 
\end{corollary}

\begin{proof}
Like B\'ona's proof of \cite[Thm.\ 3]{bona_real_2008}, our proof relies on a theorem of Darroch~\cite{darroch_distribution_1964} (see also
\cite[Thm.\ 3.25]{bona_combinatorics_2012} and \cite[Prop.\ 1]{pitman_probabilistic_1997}), which implies that if $p(t)=\sum_{m=0}^n p_m t^m$ is a polynomial that has only real roots and satisfies $p(1)>0$, then an index $m$ that maximizes $p_m$ must satisfy $|m-p'(1)/p(1)|<1$.

By Theorem~\ref{thm:roots}, the polynomial $\oQ_n(t)$ has only real roots, and by Corollary~\ref{cor:average}, $$\frac{\oQ_n'(1)}{\oQ_n(1)}
=\frac{\oQ_n'(1)}{|\QQ_n|}=\frac{3n+1}{4},$$ so the first statement follows.

On the other hand, it is well-known (see \cite[Thm.\ 5.3.1]{brenti_unimodal_1989} and \cite{bona_combinatorics_2012}) that the Narayana polynomials $\sum_{m=1}^n N(n,m) u^m$, which give the distribution of the number of plateaus by Equation~\eqref{eq:Narayana}, have only real roots. Thus, the statement regarding plateaus follows similarly Corollary~\ref{cor:average}.
\end{proof}

\subsection{Asymptotically normal distribution}

Here we prove that the distribution of each of the statistics $\asc$, $\des$ and $\plat$ on quasi-Stirling permutations is asymptotically normal. We use a result of Bender, that can be stated as follows. 

\begin{theorem}[{\cite{bender_central_1973}, see also \cite{canfield_central_1977,bona_real_2008,pitman_probabilistic_1997}}]\label{thm:bender}
Let $\{X_n\}_n$ be a sequence of random variables, where $X_n$ takes values in $[n]$. Suppose that the polynomials $g_n(t)=\sum_{m=1}^n P(X_n=m)\, t^m$ have only real roots, and that 
\begin{equation}\label{eq:var}
 \sigma_n=\sqrt{\Var(X_n)}\to\infty
\end{equation}
as $n\to\infty$. Then 
$$\frac{X_n-E(X_n)}{\sigma_n}\rightarrow N(0,1),$$ which denotes convergence in distribution to the standard normal distribution.
\end{theorem}

\begin{theorem}
The distribution of the number of descents (resp.\ ascents, plateaus) on elements of $\QQ_n$ converges 
to a normal distribution as $n\to\infty$.
\end{theorem}

\begin{proof}
In order to apply Theorem~\ref{thm:bender} to descents (equivalently, ascents) on quasi-Stirling permutations, we let $X_n$ be the number of descents of a random element of $\QQ_n$.
Then the polynomials $$g_n(t)=\sum_{m=1}^n P(X_n=m)\, t^m=\frac{\oQ_n(t)}{|\QQ_n|}$$ have only real roots by Theorem~\ref{thm:roots}.
It remains to show that Equation~\eqref{eq:var} holds.

Using that $\oQ'_n(1)+\oQ''_n(1)=\sum_{\pi\in\QQ_n} \des(\pi)^2$, we have
\begin{equation}\label{eq:varformula}
\Var(X_n)=E(X_n^2)-E(X_n)^2=\frac{\oQ'_n(1)+\oQ''_n(1)}{|\QQ_n|}-\left(\frac{3n+1}{4}\right)^2,
\end{equation}
since $E(X_n)=\oQ'_n(1)/|\QQ_n|=(3n+1)/4$ by Corollary~\ref{cor:average}.

Differentiating Equation~\eqref{eq:oQ} twice with respect to $t$, setting $t=1$, and solving for $\frac{\partial^2 \oQ}{\partial t^2}(1,z)$, we get
\begin{align*}
\frac{\partial^2 \oQ}{\partial t^2}(1,z)
&=\frac{z}{6}\,\frac{z^3\oQ(1,z)^4-4z^2\oQ(1,z)^3+6z\oQ(1,z)^2+12 z\left(\frac{\partial \oQ}{\partial t}(1,z)\right)^2+12\frac{\partial \oQ}{\partial t}(1,z)}{1-z-3z\oQ(1,z)+z^2\oQ(1,z)+3z^2\oQ(1,z)^2-z^3\oQ(1,z)^3}\\
&=\frac{1}{12z}\left(\frac{14{z}^{3}-18{z}^{2}+23z-4}{(1-4z)^{3/2}}+z+4\right),
\end{align*}
where we have used Equations~\eqref{eq:Catalan} and~\eqref{eq:oQ'} in the last equality.
Extracting the coefficient of $z^n$ in
$$\frac{\partial^2 \oQ}{\partial t^2}(1,z)+\frac{\partial \oQ}{\partial t^2}(1,z)=\frac{1}{12z}\left(\frac{14{z}^{3}-6{z}^{2}+8z-1}{(1-4z)^{3/2}}-2z+1\right),$$
we get that, for $n\ge1$,
\begin{align*}\frac{\oQ'_n(1)+\oQ''_n(1)}{n!}&=\frac{1}{12}[z^{n+1}]\frac{14{z}^{3}-6{z}^{2}+8z-1}{(1-4z)^{3/2}}\\
&=14(2n-3)\binom{2n-4}{n-2}-6(2n-1)\binom{2n-2}{n-1}+8(2n+1)\binom{2n}{n}-(2n+3)\binom{2n+2}{n+1}\\
&=\frac{27n^3 + 10n^2 - 9n - 4}{12n(n + 1)}\binom{2n-2}{n-1}.
\end{align*}
Plugging this expression and the equality $|\QQ_n|=n!C_n$ into Equation~\eqref{eq:varformula}, we obtain
$$\Var(X_n)=\frac {11{n}^{2}-6n-5}{48(2n-1)},$$
and so Equation~\eqref{eq:var} holds.
By Theorem~\ref{thm:bender}, we conclude that
$$\frac{X_n-\frac{3n+1}{4}}{\sqrt{\frac {11{n}^{2}-6n-5}{48(2n-1)}}}\rightarrow N(0,1).$$

Asymptotic normality of the distribution of the number of plateaus can be proved similarly using Equation~\eqref{eq:Narayana};
in fact, it is already known that the Narayana distribution is asymtotically normal, see for example~\cite{fulman_steins_2018}.
\end{proof}

\section{Generalization to $k$-quasi-Stirling permutations}
\label{sec:k}

In the rest of the paper, we significantly extend the results from Section~\ref{sec:des}. On the one hand, we refine them to track the joint distribution of the number of ascents, the number of descents, and the number of plateaus. On the other hand, we generalize them to a one-parameter family of permutations, by allowing $k$ copies of each element.

Gessel and Stanley proposed in~\cite{gessel_stirling_1978} a generalization of Stirling permutations by considering, for a fixed positive integer~$k$, permutations of the multiset $\{1^k,2^k,\dots,n^k\}$ (where the notation $i^k$ indicates $k$ copies of $i$) that avoid the pattern $212$. Let $\Q^k_n$ denote the set of these permutations, which we call {\em $k$-Stirling permutations}, following~\cite{janson_generalized_2011,kuba_analysis_2011}, although they are called $r$-multipermutations in~\cite{park_r-multipermutations_1994,park_inverse_1994}. 
Note that $\Q^1_n=\S_n$ and $\Q^2_n=\Q_n$. For this reason, the coefficients of the Stirling polynomials $Q_n(t)$ are called {\em second-order Eulerian numbers} in \cite{graham_concrete_1994}. 
An even more general version where each element $i$ appears an arbitrary number of times, which may be different for each $i$, was considered by Brenti~\cite{brenti_unimodal_1989}.

Generalizing the definition of quasi-Stirling permutations in an analogous manner, we define {\em $k$-quasi-Stirling permutations} as those permutations of $\{1^k,2^k,\dots,n^k\}$ that avoid the patterns $1212$ and $2121$, and denote this set by $\QQ^k_n$.
Note that $\QQ^1_n=\S_n$ and $\QQ^2_n=\QQ_n$. Viewing  permutations of $\{1^k,2^k,\dots,n^k\}$ as ordered set partitions into blocks of size $k$, the avoidance requirement is equivalent to the partition being {\em noncrossing }.
 In this section, we enumerate $k$-quasi-Stirling permutations by the number of ascents, the number of descents, and the number of plateaus.

\subsection{Bijections to trees}\label{sec:bij_k}

We present two bijections between $k$-quasi-Stirling permutations and different kinds of trees, each one extending a bijection in the literature between $k$-Stirling permutations and certain increasing trees.

In~\cite[Thm.\ 1]{janson_generalized_2011}, Janson, Kuba and Panholzer describe a bijection between $k$-Stirling permutations and {\em $(k+1)$-ary increasing trees} ---this bijection is attributed to Gessel in~\cite{park_r-multipermutations_1994}---, and they use it to study the distribution of ascents, descents and plateaus on $k$-Stirling permutations. A (vertex-labeled) $k$-ary tree is a plane rooted tree where each vertex has either $0$ or $k$ children, and each of the internal (i.e. non-leaf) vertices receives a distinct label between 1 and the number of internal vertices. 
Let $\A^k_n$ denote the set of $k$-ary trees with $n$ internal vertices. Such a tree is {\em increasing} if the label of each vertex is smaller than the labels of its children, disregarding unlabeled leaves.

Inspired by this bijection, we can construct a bijection $\psi$ between $k$-ary trees without the increasing condition, and $k$-quasi-Stirling permutations.

\begin{theorem}\label{thm:psi}
There is a natural bijection  $\psi:\A^k_n\to\QQ^k_n$.
\end{theorem}

\begin{proof}
Given a tree in $\A^k_n$, traverse its edges by following a depth-first walk from left to right, and record the label of each vertex that the path returns to; in other words, record every time that a vertex is visited except for the first time.  See Figure~\ref{fig:kary} for examples with $k=2$ and $k=3$.  Let us show that the resulting sequence belongs to $\QQ^k_n$. First, it contains $k$ copies of each element in $[n]$, since the path returns to each internal vertex once coming from each of its $k$ children. Second, it avoids the patterns $1212$ and $2121$ because, if a label $a$ appears between two readings of a label $b$ in the depth-first walk, then vertex $b$ is in the path between vertex $a$ and the root, and so all occurrences of $a$ in the sequence appear between those two occurrences of~$b$. 

To see that $\psi$ is a bijection, let us describe its inverse. Given $\pi\in\QQ^k_n$, let $b=\pi_n$, and decompose $\pi$ as $\pi=\sigma_1 b \sigma_2 b \dots \sigma_{k-1} b \sigma_k b$. The subsequences $\sigma_1,\sigma_2,\dots,\sigma_k$ must have pairwise disjoint entries, because $\pi$ avoids $1212$ and $2121$, and so each $\sigma_i$ is a $k$-quasi-Stirling permutation whose entries have been relabeled by an order-preserving function.
Then $\psi^{-1}(\pi)$ is the $k$-ary tree consisting of a root labeled $b$ with subtrees $\psi^{-1}(\sigma_1),\psi^{-1}(\sigma_2),\dots,\psi^{-1}(\sigma_k)$ from left to right, constructed recursively.
\end{proof}

For $k=2$, $\psi$ is a bijection between (vertex-labeled) binary trees and quasi-Stirling permutations. 
Interestingly, the shift of the parameter $k$ in the bijection from~\cite{janson_generalized_2011} between $(k+1)$-ary increasing trees and $k$-Stirling permutations disappears in our bijection $\psi$. 

\medskip

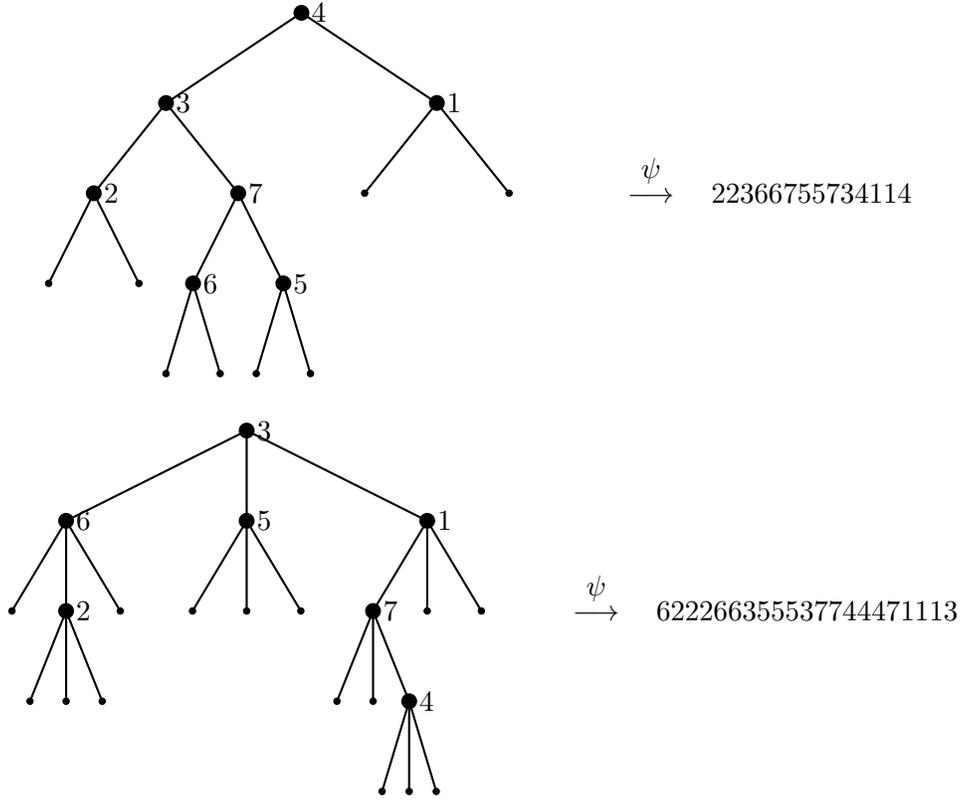
\begin{figure}[htb]
\centering
 \begin{tikzpicture}[scale=1.2]
     \draw[thick] (0,3) coordinate(d4) -- ++(-1.5,-1) coordinate(d3) -- ++(-.8,-1) coordinate(d2) -- ++(-.5,-1);
     \draw[thick] (d2) -- ++(.5,-1); 
     \draw[thick] (d3) -- ++(.8,-1) coordinate(d7) -- ++(-.5,-1)  coordinate(d6) -- ++(-.3,-1);
     \draw[thick] (d6) -- ++(.3,-1); 
     \draw[thick] (d7) -- ++(.5,-1) coordinate(d5)-- ++(-.3,-1);
     \draw[thick] (d5) -- ++(.3,-1); 
     \draw[thick] (d4) -- ++(1.5,-1) coordinate(d1) -- ++(-.8,-1);
     \draw[thick] (d1) -- ++(.8,-1); 
     \foreach \x in {1,...,7} {
        \draw[fill,very thick] (d\x) circle (2pt);
        \draw (d\x) node[right] {$\x$};
      }
     \foreach \x in {1} {
        \draw[fill] (d\x)+(-.8,-1) circle (1pt);
        \draw[fill] (d\x)+(.8,-1) circle (1pt);
	}
     \foreach \x in {2} {
        \draw[fill] (d\x)+(-.5,-1) circle (1pt);
        \draw[fill] (d\x)+(.5,-1) circle (1pt);
	}
     \foreach \x in {6,5} {
        \draw[fill] (d\x)+(-.3,-1) circle (1pt);
        \draw[fill] (d\x)+(.3,-1) circle (1pt);
	}
      \draw (3.5,1) node[right] {$\longrightarrow \quad 22366755734114$};
      \draw (3.65,1) node[above right]{$\psi$};
      \end{tikzpicture}
\bigskip
      
 \begin{tikzpicture}[scale=1.2]
     \draw[thick] (0,3) coordinate(d3) -- (-2,2) coordinate(d6) -- (-2,1) coordinate(d2) -- (-2,0);
     \draw[thick] (d6) -- ++(-.6,-1);      \draw[thick] (d6) -- ++(0.6,-1); 
     \draw[thick] (d2) -- ++(-.4,-1);      \draw[thick] (d2) -- ++(0.4,-1); 
     \draw[thick] (d3) -- (0,2) coordinate(d5) -- (0,1); 
     \draw[thick] (d5) -- ++(-.6,-1);      \draw[thick] (d5) -- ++(0.6,-1); 
     \draw[thick] (d3) -- (2,2) coordinate(d1) -- (1.4,1) coordinate(d7) -- (1.8,0) coordinate(d4) -- (1.8,-1);
     \draw[thick] (d1) -- ++(0,-1);      \draw[thick] (d1) -- ++(0.6,-1); 
     \draw[thick] (d7) -- ++(-.4,-1);      \draw[thick] (d7) -- ++(0,-1); 
     \draw[thick] (d4) -- ++(-.3,-1);      \draw[thick] (d4) -- ++(0.3,-1); 
     \foreach \x in {1,...,7} {
        \draw[fill,very thick] (d\x) circle (2pt);
        \draw (d\x) node[right] {$\x$};
      }
     \foreach \x in {1,5,6} {
        \draw[fill] (d\x)+(-.6,-1) circle (1pt);
        \draw[fill] (d\x)+(0,-1) circle (1pt);
        \draw[fill] (d\x)+(.6,-1) circle (1pt);
	}
     \foreach \x in {2,7} {
        \draw[fill] (d\x)+(-.4,-1) circle (1pt);
        \draw[fill] (d\x)+(0,-1) circle (1pt);
        \draw[fill] (d\x)+(.4,-1) circle (1pt);
	}
     \foreach \x in {4} {
        \draw[fill] (d\x)+(-.3,-1) circle (1pt);
        \draw[fill] (d\x)+(0,-1) circle (1pt);
        \draw[fill] (d\x)+(.3,-1) circle (1pt);
	}
      \draw (3.5,1) node[right] {$\longrightarrow \quad 622266355537744471113$};
      \draw (3.65,1) node[above right]{$\psi$};
      \end{tikzpicture}
      \caption{Examples of the bijection $\psi:\A^k_n\to\QQ^k_n$ for $k=2$ (above) and $k=3$ (below).}
      \label{fig:kary}
\end{figure}

In order to study ascents, descents and plateaus, we introduce a second bijection $\phi$ between trees and $k$-quasi-Stirling permutations that will be more suitable to track these statistics. It extends a construction of Kuba and Panholzer \cite[Thm.\ 2.2]{kuba_analysis_2011} that, with a slight modification, yields a bijection between $k$-Stirling permutations and certain modified increasing trees that we describe next.

For $k\ge2$, let $\T^k_n$ be the set of edge-labeled plane rooted trees with $n$ edges, where every node other than the root has $k-2$ unlabeled half-edges which serve as walls, separating its children into $k-1$ (possibly empty) compartments. By definition, $\T^2_n=\T_n$. A tree in $\T^3_n$ is drawn on the left of Figure~\ref{fig:phi}. We call trees in $\T^k_n$ {\em compartmented  trees}.
Denote by $\I^k_n\subseteq\T^k_n$ the subset of those trees whose labels along any path from the root to a leaf are increasing. 

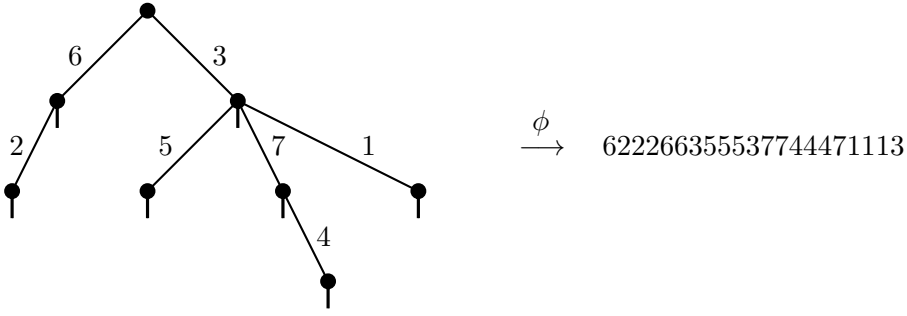
\begin{figure}[htb]
\centering
 \begin{tikzpicture}[scale=1.2]
     \draw[thick] (0,3) coordinate(d0) -- (-1,2) coordinate(d6) -- (-1.5,1) coordinate(d2);
     \draw[thick] (d0) -- (1,2) coordinate(d3) -- (0,1) coordinate(d5); 
     \draw[thick] (d3) -- (1.5,1) coordinate(d7) -- (2,0) coordinate(d4);
     \draw[thick] (d3) -- (3,1) coordinate(d1);
      \draw[fill,very thick] (d0) circle (2pt);
      \foreach \x in {1,...,7} {
        \draw[fill,very thick] (d\x) circle (2pt) -- ++(0,-0.3);
      }
      \foreach \x in {5,6} {
        \draw (d\x)+(.2,.5) node {$\x$};
      }
      \foreach \x in {2} {
        \draw (d\x)+(.05,.5) node {$\x$};
      }
      \foreach \x in {1} {
        \draw (d\x)+(-.55,.5) node {$\x$};
      }     
      \foreach \x in {3} {
        \draw (d\x)+(-.2,.5) node {$\x$};
      }     
      \foreach \x in {4,7} {
        \draw (d\x)+(-.05,.5) node {$\x$};
      }
      \draw (4,1.5) node[right] {$\longrightarrow \quad 622266355537744471113$};
      \draw (4.15,1.5) node[above right]{$\phi$};
      \end{tikzpicture}
      \caption{An example of the bijection $\phi:\T^3_n\to\QQ^3_n$.}
      \label{fig:phi}
\end{figure}

We are now ready to describe the bijection $\phi$ between $\T^k_n$ and $\QQ^k_n$. For $k=2$, $\phi$ coincides with $\varphi:\T_n\to\QQ_n$ described in Section~\ref{sec:bij}. When restricted to increasing trees, $\phi$ becomes a bijection between $\I^k_n$ and $\Q^k_n$, which is a version of~\cite[Thm.\ 2.2]{kuba_analysis_2011}. 

\begin{theorem}\label{thm:phi}
There is a natural bijection  $\phi:\T^k_n\to\QQ^k_n$.
\end{theorem}

\begin{proof}
Given a tree $T\in\T^k_n$, first label the half-edges at each node $v$ with the label of the edge between $v$ and its parent. Then traverse the edges and half-edges of $T$ following a depth-first walk from left to right, and record their labels as they are traversed, with the rule that the label of each half-edge only contributes once.
Let $\phi(T)$ be the resulting sequence of recorded labels; see Figure~\ref{fig:phi} for an example.
We claim that $\phi(T)\in\QQ^k_n$. Indeed, each element in $[n]$ appears $k$ times: twice from traversing the edge with that label, and $k-2$ times from traversing the half-edges with that label. Additionally, the sequence does not contain the patterns $1212$ and $2121$ because, if a label $a$  appears between two readings of a label $b$ in the depth-first walk, then the whole subtree containing the edge labeled $a$ (which includes all occurrences of $a$) must be read between those two occurrences of~$b$. 

Next we show that $\phi$ is a bijection by describing its inverse. Given $\pi\in\QQ^k_n$, let $a=\pi_1$, and decompose $\pi$ according to the occurrences of $a$ as
$\pi=a\sigma_1 a \sigma_2 a \dots a \sigma_{k-1} a \sigma_k$. The subsequences $\sigma_1,\sigma_2,\dots,\sigma_k$ must have pairwise disjoint entries, because $\pi$ avoids $1212$ and $2121$, and so each $\sigma_i$ is a $k$-quasi-Stirling permutation with relabeled entries.
We obtain $\phi^{-1}(\pi)$ as follows. First, place an edge with label $a$ from the root to its leftmost child. In the $k-1$ compartments
hanging from that child, place the subtrees $\phi^{-1}(\sigma_1),\dots,\phi^{-1}(\sigma_{k-1})$ from left to right, constructed recursively.
Finally, place the subtree $\phi^{-1}(\sigma_k)$ hanging from the root of $\phi^{-1}(\pi)$, to the right of the initially placed edge.
\end{proof}

Either of the bijections $\psi$ or $\phi$ allows us to easily count the number of $k$-quasi-Stirling permutations. We denote the {\em $k$-Catalan numbers} (see~\cite[pp.\ 168--173]{stanley_enumerative_1999}) by $$C_{n,k}=\frac{1}{(k-1)n+1}\binom{kn}{n}.$$

\begin{theorem}\label{thm:QQkn}
For $n\ge1$ and $k\ge1$,
$$|\QQ^k_n|=\frac{(kn)!}{((k-1)n+1)!}=n!\,C_{n,k}.$$
\end{theorem}

\begin{proof}
Consider the EGF $\oG(z)=\sum_{n\ge0}|\QQ^k_n|z^n/n!$.
By Theorem~\ref{thm:phi}, $|\QQ^k_n|=|\T^k_n|$. We enumerate compartmented trees by giving a recursive description of the class $\T^k=\bigcup_{n\ge0} \T^k_n$.
Every non-empty tree in $\T^k$ consists of a root having a sequence of children, where each of these children plays itself the role of a root of a sequence of $k-1$ trees from $\T^k$, one in each of the compartments. This description translates into the equation
$$\oG(z)=\frac{1}{1-z\oG(z)^{k-1}},$$
which is equivalent to $\oG(z)-z\oG(z)^k=1$.

Alternatively, using Theorem~\ref{thm:psi}, $|\QQ^k_n|=|\A^k_n|$, so one can enumerate $k$-ary trees instead. The obvious recursive description of such trees yields the equation $\oG(z)=1+z\oG(z)^k$ for their EGF.

Applying the Lagrange inversion formula to $F(z):=\oG(z)-1$, which satisfies $F(z)=z(1+F(z))^k$, we get, for $n\ge1$,
$$\frac{|\QQ^k_n|}{n!}=[z^n]F(z)=\frac{1}{n}[z^{n-1}](1+z)^{kn}=\frac{1}{n}\binom{kn}{n-1}=C_{n,k}.$$
\end{proof}

\subsection{Ascents, descents and plateaus on $k$-quasi-Stirling permutations}

We are now ready to generalize Theorem~\ref{thm:main} to give an equation for the refined $k$-quasi-Stirling polynomials
$$\oP^{(k)}_n(q,t,u)=\sum_{\pi\in\QQ^k_n} q^{\asc(\pi)}t^{\des(\pi)} u^{\plat(\pi)},$$
and their EGF
$$\oP^{(k)}(q,t,u;z)=\sum_{n\ge0}\oP^{(k)}_n(q,t,u)\frac{z^n}{n!}.$$
For $k=2$, we have $\oP^{(2)}(1,t,1;z)=\oQ(t,z)$ by definition, which we computed in Theorem~\ref{thm:main}. On the other hand, an expression for $\oP^{(2)}(1,1,u;z)$ is given by Equation~\eqref{eq:Narayana}.

In order to track ascents, it will be useful to consider the homogenization of the Eulerian polynomials, 
$$\hat{A}_n(q,t)=\sum_{\pi\in\S_n} q^{\asc(\pi)} t^{\des(\pi)},$$ and their EGF
\begin{align*}\hat{A}(q,t;z)&=\sum_{n\ge0} \hat{A}_n(q,t) \frac{z^n}{n!}=
1+\sum_{n\ge1} A_n(t/q) q^{n+1} \frac{z^n}{n!} = 1+q(A(t/q,qz)-1)\\
&=1-q+\frac{q(q-t)}{q-te^{(q-t)z}},
\end{align*}
using Equation~\eqref{eq:A_GF}.

\begin{theorem}\label{thm:main_plat_k}
Fix $k\ge1$. The EGF of $k$-quasi-Stirling permutations by the number of ascents, the number of descents, and the number of plateaus satisfies the implicit equation $\oP^{(k)}(q,t,u;z)=\hat{A}(q,t,z(\oP^{(k)}(q,t,u;z)-1+u)^{k-1})$, that is,
\begin{equation}\label{eq:oP}
\oP^{(k)}(q,t,u;z)=1-q+\frac{q(q-t)}{q-t\,\exp\left((q-t)z(\oP^{(k)}(q,t,u;z)-1+u)^{k-1}\right)}.
\end{equation}
In particular, for $n\ge1$, its coefficients satisfy 
\begin{equation}
\oP^{(k)}_n(q,t,u)=\frac{n!}{(k-1)n+1}\,[z^n]\left(\hat{A}(q,t;z)-1+u\right)^{(k-1)n+1}.
\end{equation}
\end{theorem}

Our proof relies on the bijection $\phi$. The first step is to analyze how the statistics $\asc$, $\des$ and $\plat$ are transformed by this bijection.
To this end, we extend the notion of cyclic descents to compartmented trees.
Let $T\in\T^k_n$, and let $v$ a vertex of $T$. If $v$ is not the root, define $\cdes(v)$ to be the number of cyclic descents of the sequence obtained by listing the labels of the edges and half-edges incident to $v$ in counterclockwise order, where the half-edges are given the label of the edge between $v$ and its parent. Equivalently, if this label is $\ell$, and the labels of the edges between $v$ and its children in the $j$-th compartment from the left are $a_{j,1},a_{j,2},\dots,a_{j,d_j}$ from left to right, for $1\le j\le k-1$, then 
\begin{equation}\label{eq:cdes_k}
\cdes(v)=\cdes(\ell a_{1,1}\dots a_{1,d_1}\ell a_{2,1}\dots a_{2,d_2}\ell \dots \ell a_{k-1,1}\dots a_{k-1,d_{k-1}}).
\end{equation} 
If $v$ is the root of $T$, define $\cdes(v)$ to be the number of descents of the sequence obtained by listing the labels of the edges incident to $v$ from left to right. Finally, define the number of cyclic descents of $T$ to be $$\cdes(T)=\sum_v \cdes(v),$$ where the sum ranges over all the vertices $v$ of $T$.
For example, if $T$ is the tree on the left of Figure~\ref{fig:phi}, then $\cdes(T)=2+1+0+2+0+1+0+0=6$, where the two vertices with $\cdes(v)=2$ are the root and the vertex with 3 children.

Switching the direction of the inequalities, we define the number of {\em cyclic ascents} of a sequence of positive integers as
$\casc(\pi_1\pi_2\dots\pi_r)=|\{i\in[r]: \pi_i<\pi_{i+1}\}|$, with the convention $\pi_{r+1}:=\pi_1$.
For a non-root vertex $v$ of $T\in\T^k_n$, define $\casc(v)$ in analogy to Equation~\eqref{eq:cdes_k}, replacing $\cdes$ with $\casc$. If $v$ is the root of $T$, define $\casc(v)$ to be the number of ascents of the sequence obtained by listing the labels of the edges incident to $v$ from left to right. The number of cyclic ascents of $T$ is then defined as $\casc(T)=\sum_v \casc(v)$, summing over all the vertices of $T$.

Recall that the children of each non-root vertex of $T\in\T^k_n$ are separated into $k-1$ (possibly empty) compartments. Denote by $\emp(T)$ the total number of empty compartments of $T$. For example, if $T$ is the tree on the left of Figure~\ref{fig:phi}, then $\emp(T)=10$.

\begin{lemma}\label{lem:cdes_k}
The bijection $\phi:\T^k_n\to\QQ^k_n$ has the following property: if $T\in\T^k_n$ and $\pi=\phi(T)\in\QQ^k_n$, then $$\asc(\pi)=\casc(T),  \quad \des(\pi)=\cdes(T)  \quad \text{and} \quad \plat(\pi)=\emp(T).$$
\end{lemma}

\begin{proof}
The proof of the equality $\des(\pi)=\cdes(T)$ is similar to the proof of Lemma~\ref{lem:cdes}, the only difference being the contribution to $\pi=\phi(T)$ of the half-edges at each non-root vertex $v$ of $T$. These half-edges are labeled with the label $\ell$ of the edge $e$ between $v$ and its parent, and then traversed by the depth-first walk in the definition of $\phi$. Each such half-edge $h$ creates an entry $\ell$ in $\pi$, which is inserted between the label of the edge or half-edge incident to $v$ immediately to the left of $h$ (or, in its absence, the edge $e$), and the label of the edge or half-edge immediately to the right of $h$ (or, in its absence, the edge $e$).
With these additional entries in $\pi=\phi(T)$, the contribution to $\des(\pi)$ of all the visits to $v$ of the depth-first walk around $T$ equals $\cdes(v)$, as defined in Equation~\eqref{eq:cdes_k}. 
If $v$ is the root of $T$, its contribution to $\des(\pi)$ is also $\cdes(v)$, like in the proof of Lemma~\ref{lem:cdes}. Adding the contributions of all the vertices of $T$, we see that $\cdes(T)=\sum_v \cdes(v)=\des(\pi)$.

The equality $\asc(\pi)=\casc(T)$ is proved analogously by symmetry.

Finally, to show that $\plat(\pi)=\emp(T)$, note that a plateau in $\pi=\phi(T)$ occurs when two edges or half-edges (or one of each) of $T$ with the same label are traversed one immediately after the other by the depth-first walk on $T$, which happens precisely at empty compartments of $T$. 
\end{proof}

By Lemma~\ref{lem:cdes_k}, $\oP^{(k)}(q,t,u;z)$ is also the EGF for compartmented trees by the number of cyclic ascents, the number of cyclic descents, and the number of empty compartments: 
\begin{equation}\label{eq:oPT}
\oP^{(k)}(q,t,u;z)=\sum_{n\ge0} \sum_{T\in\T^k_n} q^{\casc(T)}t^{\cdes(T)}u^{\emp(T)} \frac{z^n}{n!}.
\end{equation}

Paralleling the proof of Theorem~\ref{thm:main}, we will consider a recursive description of compartmented trees. Let ${\T^k_n}'\subseteq\T^k_n$ be the subset of trees whose root has exactly one child, and let 
\begin{equation}\label{eq:Rkn}
R^{(k)}_n(q,t,u)=\sum_{T'\in{\T^k_n}'} q^{\casc(T')-1}t^{\cdes(T')-1}u^{\emp(T')}.
\end{equation}
Let $\T^k=\bigcup_{n\ge0}\T^k_n$ and ${\T^k}'=\bigcup_{n\ge1}{\T^k_n}'$. 
The next lemma generalizes Lemma~\ref{lem:T'}.

\begin{lemma}\label{lem:T'_k}
We have
$$\sum_{n\ge1} R^{(k)}_n(q,t,u) \frac{z^n}{n!}=z\,(\oP^{(k)}(q,t,u;z)-1+u)^{k-1}.$$
\end{lemma}

\begin{proof}
In the proof of Lemma~\ref{lem:T'}, we showed that every tree $T'\in\T'_n$ can be obtained as $T'=T^{|n}+j$ for a unique $j\in\{0,1,\dots,n-1\}$ and a unique $T\in\T_{n-1}$.

Here we generalize this construction to compartmented trees. 
Using the notation from~\cite{flajolet_analytic_2009}, consider the class $\Seq_{k-1}(\T^{k})$, which consists of sequences of $k-1$ trees from $\T^{k}$ whose edges have been relabeled with distinct labels from $1$ up to the total number of edges, so that the relative order of labels within each tree is preserved. If $\vec{T}\in \Seq_{k-1}(\T^{k})$ is a relabeling of the tuple $(T_1,\dots,T_{k-1})$, where $T_j\in\T^{k}$ for all $j$, define 
\begin{equation}\label{eq:cdesvecT}
\cdes(\vec{T})=\sum_{j=1}^{k-1}\cdes(T_j),\quad \casc(\vec{T})=\sum_{j=1}^{k-1}\casc(T_j),\quad
\emp(\vec{T})=\sum_{j=1}^{k-1}\emp(T_j)+|\{j:T_j \text{ is empty}\}|,
\end{equation}
and let $|\vec{T}|$ denote the total number of edges.
Using Equation~\eqref{eq:oPT}, the corresponding multivariate EGF is
\begin{equation}\label{eq:EGFSeq}
\sum_{\vec{T}\in\Seq_{k-1}(\T^{k})}  q^{\casc(\vec{T})}t^{\cdes(\vec{T})}u^{\emp(\vec{T})}\frac{z^{|\vec{T}|}}{|\vec{T}|!} = (\oP^{(k)}(q,t,u;z)-1+u)^{k-1}.
\end{equation}

Given $\vec{T}\in \Seq_{k-1}(\T^{k})$ with $|\vec{T}|=n-1$, construct a new tree $\vec{T}^{|n}\in{\T^k_n}'$ as follows: 
combine the $k-1$ trees in $\vec{T}$ by identifying their roots into a common vertex $v_0$, placing the trees from left to right with $k-2$ half-edges at $v_0$ separating them, and attach an edge with label $n$ from $v_0$ to a new root vertex. 

Using the definitions from Equation~\eqref{eq:cdesvecT}, we claim that 
\begin{equation}\label{eq:cdesvecTplus1}
\cdes(\vec{T}^{|n})=\cdes(\vec{T})+1, \quad \casc(\vec{T}^{|n})=\casc(\vec{T})+1, \quad \emp(\vec{T}^{|n})=\emp(\vec{T}).
\end{equation}
To prove the first equality, suppose that $\vec{T}$ is a relabeling of the tuple $(T_1,\dots,T_{k-1})$. If $v$ is a non-root vertex of $T_j$ for some $j$, then $\cdes(v)$ is the same in $T_j$ as it is in $\vec{T}^{|n}$. On the other hand, assuming that the edges incident to the root of $T_j$ have labels $a_{j,1},a_{j,2},\dots,a_{j,d_j}$ from left to right, the root of $T_j$ contributes $\des(a_{j,1}a_{j,2}\dots a_{j,d_j})$ to $\cdes(\vec{T})$ for each $j$, whereas the corresponding vertex $v_0$ in $\vec{T}^{|n}$ contributes
\begin{align*}
\cdes(v_0)&=\cdes(n a_{1,1}\dots a_{1,d_1}n a_{2,1}\dots a_{2,d_2}n \dots n a_{k-1,1}\dots a_{k-1,d_{k-1}})=
\sum_{j=1}^{k-1} \cdes(n a_{j,1}a_{j,2}\dots a_{j,d_j})\\
&=\sum_{j=1}^{k-1} \des(a_{j,1}a_{j,2}\dots a_{j,d_j})
\end{align*}
to $\cdes(\vec{T}^{|n})$. Finally, the new root of $\vec{T}^{|n}$ contributes an additional cyclic descent, proving the first equality in Equation~\eqref{eq:cdesvecTplus1}.
A symmetric argument proves the analogous statement for $\casc$. The third equality in Equation~\eqref{eq:cdesvecTplus1} is clear by construction, since each empty $T_j$ contributes an additional empty compartment in $\vec{T}^{|n}$.

For $T'\in{\T^k_n}'$ and $i\in\{0,1,\dots,n-1\}$, let $T'+i\in{\T^k_n}'$ be the tree obtained by adding $i$ modulo $n$ to each label of $T'$. 
As in the proof of Lemma~\ref{lem:T'}, this operation preserves the number of cyclic descents, that is, $\cdes(T'+i)=\cdes(T')$ for all $T'\in{\T^k_n}'$.
In addition, $\casc(T'+i)=\casc(T')$ by symmetry, and $\emp(T'+i)=\emp(T')$ because adding $i$ does not change the underlying unlabeled tree, and in particular its number of empty compartments.

Noting that every tree $T'\in{\T^k_n}'$ can be obtained as $T'=\vec{T}^{|n}+i$ for a unique $\vec{T}\in \Seq_{k-1}(\T^{k})$ with $|\vec{T}|=n-1$, and a unique $i\in\{0,1,\dots,n-1\}$, we have
\begin{align*}
R^{(k)}_n(q,t,u)&=\sum_{T'\in{\T^k_n}'} q^{\casc(T')-1}t^{\cdes(T')-1}u^{\emp(T')}\\
&=\sum_{\substack{\vec{T}\in\Seq_{k-1}(\T^{k}) \\ |\vec{T}|=n-1}} \sum_{i=0}^{n-1} q^{\casc(\vec{T}^{|n}+i)-1}t^{\cdes(\vec{T}^{|n}+i)-1}u^{\emp(\vec{T}^{|n}+ )}\\
&=\sum_{\substack{\vec{T}\in\Seq_{k-1}(\T^{k}) \\ |\vec{T}|=n-1}} n\, q^{\casc(\vec{T}^{|n})-1}t^{\cdes(\vec{T}^{|n})-1}u^{\emp(\vec{T}^{|n})}\\
&=n\sum_{\substack{\vec{T}\in\Seq_{k-1}(\T^{k}) \\ |\vec{T}|=n-1}} q^{\casc(\vec{T})}t^{\cdes(\vec{T})}u^{\emp(\vec{T})}
\end{align*}
for every $n\ge1$. Multiplying both sides by $z^n/n!$, summing over $n$, and using Equation~\eqref{eq:EGFSeq}, we get
$$\sum_{n\ge1} R^{(k)}_n(q,t,u) \frac{z^n}{n!}=
\sum_{n\ge1} \sum_{\substack{\vec{T}\in\Seq_{k-1}(\T^{k}) \\ |\vec{T}|=n-1}} q^{\casc(\vec{T})}t^{\cdes(\vec{T})}u^{\emp(\vec{T})} \frac{z^n}{(n-1)!}
=z\,(\oP^{(k)}(q,t,u;z)-1+u)^{k-1}.$$
\end{proof}

\begin{proof}[Proof of Theorem~\ref{thm:main_plat_k}]
Recall that, by Equation~\eqref{eq:oPT}, $\oP^{(k)}(q,t,u;z)$ is the EGF for the class $\T^k$ with respect to $\casc$, $\cdes$ and $\emp$. 

Every $T\in\T^k$ can be decomposed as a sequence of non-empty trees $T'_1,T'_2,\dots,T'_r\in{\T^k}'$  with a common root, denoted by $v_0$, with edges relabeled from $1$ up to the number of edges of $T$ as usual. In this decomposition, $\emp(T)=\sum_{i=1}^r \emp(T'_j)$. 
As in the proof of Theorem~\ref{thm:main}, if the edges incident to $v_0$ in $T$ have labels $a_1,a_2,\dots,a_r$ from left to right, then
$$\cdes(T)=\sum_{i=1}^r(\cdes(T'_i)-1)+\des(a_1a_2\dots a_r),$$ and similarly
$$\casc(T)=\sum_{i=1}^r(\casc(T'_i)-1)+\asc(a_1a_2\dots a_r).$$
It follows that 
\begin{multline*}
\sum_{T\in\T^k_n} q^{\casc(T)}t^{\cdes(T)}u^{\emp(T)} \\
= \sum_{\substack{\{B_1,B_2,\dots,B_r\}\\ \text{partition of $[n]$}}} 
\left(\sum_{\pi\in\S_r} q^{\asc(\pi)} t^{\des(\pi)}\right)\cdot\prod_{i=1}^r \left(\sum_{T'_i\in\T'_{|B_i|}} q^{\casc(T'_i)-1}t^{\cdes(T'_i)-1}u^{\emp(T'_i)} \right).
\end{multline*}
Using the notation $R^{(k)}_n(q,t,u)$ from Equation~\eqref{eq:Rkn}, we can rewrite this equation as
$$\oP^{(k)}_n(q,t,u)=\sum_{\substack{\{B_1,B_2,\dots,B_r\}\\ \text{partition of $[n]$}}} \hat{A}_r(q,t)\cdot \prod_{i=1}^r R^{(k)}_{|B_i|}(q,t,u).$$
By the Compositional Formula, and using Lemma~\ref{lem:T'_k},
\begin{align*}
\oP^{(k)}(q,t,u;z)&=\sum_{n\ge0}\oP^{(k)}_n(q,t,u)\frac{z^n}{n!}=1+\sum_{r\ge1}\hat{A}_r(q,t)\frac{1}{r!} \left(\sum_{b\ge1} R^{(k)}_{b}(q,t,u)\frac{z^{b}}{b!}\right)^r\\
&=1+\sum_{r\ge1} \hat{A}_r(q,t) \frac{1}{r!} \left(z\,(\oP^{(k)}(q,t,u;z)-1+u)^{k-1}\right)^r  \\
&=\hat{A}(q,t,z(\oP^{(k)}(q,t,u;z)-1+u)^{k-1}),
\end{align*}
proving Equation~\eqref{eq:oP}.

To extract the coefficient of $z^n$ in $\oP^{(k)}(q,t,u;z)$, first write $z=y^{k-1}$, so that
$$\oP^{(k)}(q,t,u;y^{k-1})=\hat{A}(q,t,y^{k-1}(\oP^{(k)}(q,t,u;y^{k-1})-1+u)^{k-1}).$$
Then the generating function $F(y):=F(q,t,u;y):=y\left(\oP^{(k)}(q,t,u;y^{k-1})-1+u\right)$ satisfies the equation
$$F(y)=y\left(\hat{A}(q,t,F(y)^{k-1})-1+u\right).$$
Applying Lagrange's inversion formula to $F(y)$ yields, for $n\ge1$,
\begin{align*}
\oP^{(k)}_n(q,t,u)&=n!\,[z^n]\oP^{(k)}(q,t,u;z)=n!\,[y^{(k-1)n+1}]F(y)\\
&=\frac{n!}{(k-1)n+1}\,[y^{(k-1)n}]\left(\hat{A}(q,t;y^{k-1})-1+u\right)^{(k-1)n+1}\\
&=\frac{n!}{(k-1)n+1}\,[z^n]\left(\hat{A}(q,t;z)-1+u\right)^{(k-1)n+1}.
\end{align*}
\end{proof}

\subsection{Ascents, descents and plateaus on $k$-Stirling permutations}

B\'ona proves in \cite[Prop.\ 1]{bona_real_2008} that, on the set $\Q_n$ of Stirling permutations, the statistics $\asc$, $\des$ and $\plat$ are equidistributed. Janson generalizes this result in \cite[Thm.\ 2.1]{janson_plane_2008}, proving, using a probabilistic argument, that the polynomials
\begin{equation}\label{eq:Pn}
 P_n(q,t,u)=\sum_{\pi\in\Q_n} q^{\asc(\pi)} t^{\des(\pi)} u^{\plat(\pi)}
\end{equation}
are symmetric in the three variables, and he suggests in \cite[Sec.\ 3]{janson_plane_2008} that the study of these polynomials would be interesting.
In~\cite{haglund_stable_2012}, Haglund and Visontai prove that these polynomials are {\em stable}, meaning that they do not vanish when all the variables have a positive imaginary part.

In this subsection we present a nice differential equation satisfied by the EGF of these polynomials, which we denote by
$$P(q,t,u;z)=\sum_{n\ge0}P_n(q,t,u)\frac{z^n}{n!},$$
and more generally, by the analogous multivariate EGF for $k$-Stirling permutations:
$$P^{(k)}(q,t,u;z)=\sum_{n\ge0}\sum_{\pi\in\Q^k_n} q^{\asc(\pi)} t^{\des(\pi)} u^{\plat(\pi)}\frac{z^n}{n!}.$$
We give two derivations of the differential equation. The first one uses the same techniques as in the proofs of Lemma~\ref{lem:T'_k} and Theorem~\ref{thm:main_plat_k}, including our interpretation of ascents, descents and plateaus in permutations as cyclic ascents, cyclic descents and empty compartments in compartmented trees. The second one is a more direct, conceptual proof.

\begin{theorem}\label{thm:Qplat_k}
The EGF $P(z):=P^{(k)}(q,t,u;z)$ of $k$-Stirling permutations by the number of ascents, the number of descents, and the number of plateaus satisfies the differential equation
$$P'(z)=(P(z)-1+q)(P(z)-1+t)(P(z)-1+u)^{k-1},$$
with initial condition $P(0)=1$.
\end{theorem}

\begin{proof}[First proof]
As discussed in Section~\ref{sec:bij_k}, the map $\phi$ restricts to a bijection between increasing compartmented trees $\I^k_n$ and $k$-Stirling permutations $\Q^k_n$. 
Thus, by Lemma~\ref{lem:cdes_k}, $P^{(k)}(q,t,u;z)$ is the EGF of increasing compartmented trees where $q$, $t$ and $u$ mark the number of cyclic ascents, the number of cyclic descents, and the number of empty compartments, respectively.

Next we adapt the construction from the proof of Lemma~\ref{lem:T'_k} to increasing trees.
Let $\I^k=\bigcup_{n\ge0}\I^k_n$ be the class of increasing compartmented trees, and let ${\I^k}'\subseteq\I^k$ be the subclass of those trees whose root has exactly one child. Every tree in ${\I^k}'$ can be obtained from a sequence $\vec{T}\in\Seq_{k-1}(\I^k)$, consisting of $k-1$ trees $T'_1,\dots,T'_{k-1}\in\I^k$ with edges relabeled as usual, by identifying the $k-1$ roots into a common vertex $v_0$, placing $k-2$ half-edges at $v_0$ to separate these trees into $k-1$ compartments, adding $1$ to every edge label, and attaching a new edge with label $1$ from $v_0$ to a new root vertex. Denote the resulting tree by $\vec{T}^{|1}\in{\I^k}'$, and observe that,
with the notation from the proof of Lemma~\ref{lem:T'_k} and letting $n=|\vec{T}|+1$, we have $\vec{T}^{|1}=\vec{T}^{|n}+1$. 
Since adding $1$ modulo $n$ to the labels preserves the statistics $\cdes$, $\casc$ and $\emp$, it follows from Equation~\eqref{eq:cdesvecTplus1} that $\cdes(\vec{T}^{|1})=\cdes(\vec{T}^{|n})=\cdes(\vec{T})+1$, $\casc(\vec{T}^{|1})=\casc(\vec{T}^{|n})=\casc(\vec{T})+1$, and $\emp(\vec{T}^{|1})=\emp(\vec{T}^{|n})=\emp(\vec{T})$. The number of edges of $\vec{T}^{|1}$ is $|\vec{T}|+1$.

Noting that every $T'\in{\I^k}'$ is obtained as $T'=\vec{T}^{|1}$ for a unique $\vec{T}\in\Seq_{k-1}(\I^k)$, and using Equation~\eqref{eq:EGFSeq},
we deduce that the refined EGF for the class ${\I^k}'$ is 
\begin{align}\nonumber
\sum_{T'\in{\I^k}'}  q^{\casc(\vec{T'})-1}t^{\cdes(\vec{T'})-1}u^{\emp(\vec{T'})}\frac{z^{|T'|}}{|T'|!}&
=\sum_{\vec{T}\in\Seq_{k-1}(\T^{k})}  q^{\casc(\vec{T})}t^{\cdes(\vec{T})}u^{\emp(\vec{T})}\frac{z^{|\vec{T}|+1}}{(|\vec{T}|+1)!}\\
&=\int_0^z(P(y)-1+u)^{k-1}\,dy.
\label{eq:Pplat}
\end{align}

The decomposition from the proof of Theorem~\ref{thm:main_plat_k}, expressing trees in $\T^k$ as sequences of trees in ${\T^k}'$, restricts straightforwardly to increasing trees: every tree in $ \I^k$ can be decomposed as a sequence of trees in ${\I^k}'$ with a common root, relabeling as usual. The same argument, using now Equation~\eqref{eq:Pplat}, implies that 
$$P(z)=\hat{A}\left(q,t,\int_0^z(P(y)-1+u)^{k-1}\,dy\right)=1-q+\frac{q(q-t)}{q-te^{(q-t)\int_0^z(P(y)-1+u)^{k-1}\,dy}}.$$
Rewriting this equation as 
$$(q-t)\int_0^z(P(y)-1+u)^{k-1}\,dy=\ln(P(z)-1+t)-\ln(P(z)-1+q)+\ln(q)-\ln(t)$$
and differentiating with respect to $z$ gives the stated differential equation.
\end{proof}

\begin{proof}[Second proof]
By looking at the positions of the $1$s, every non-empty $k$-Stirling permutation $\pi\in\Q^k_n$ can be decomposed uniquely as 
$\pi=\sigma_0 1\sigma_1 1\dots 1\sigma_{k}$. The entries in each $\sigma_i$ must be pairwise disjoint, since otherwise the repeated entry together with a $1$ in between would form an occurrence of $212$ in $\pi$. Thus, each $\sigma_i$ is itself a $k$-Stirling permutation with relabeled entries. Conversely, any $(k+1)$-tuple of $k$-Stirling permutations can be used to construct a new Stirling permutation by relabeling the entries with $\{2^k,\dots,n^k\}$ so that their relative order is preserved, and using the relabeled permutations as the subsequences $\sigma_0,\sigma_1,\dots,\sigma_{k}$ above.

Next, we examine how ascents, descents, and plateaus behave under this decomposition. Note that 
$\des(\pi)=\sum_{i=0}^{k} \des(\sigma_i)$, unless $\sigma_{k}$ is empty, in which case $\pi$ has an extra descent.
Similarly, 
$\asc(\pi)=\sum_{i=0}^{k} \asc(\sigma_i)$, unless $\sigma_0$ is empty, in which case $\pi$ has an extra ascent.
On the other hand, 
$$\plat(\pi)=\sum_{i=0}^{k} \plat(\sigma_i)+|\{i\in[k-1]:\sigma_i\text{ is empty}\}|.$$ 
It follows that
$$P'(z)=(P(z)-1+q)(P(z)-1+t)(P(z)-1+u)^{k-1}.$$
The initial condition $P(0)=1$ comes from the empty permutation, completing the proof.

We remark that, through the bijection $\phi$, this decomposition of $k$-Stirling permutations is equivalent to the decomposition of non-empty increasing compartmented trees $T\in\I^k$ into $k+1$ trees $T_0,T_1,\dots,T_k\in\I^k$, obtained by splitting $T$ at the edge with label $1$ and at the half-edges of its lower endpoint, as shown in Figure~\ref{fig:tree_decomp}.
\end{proof}

\begin{figure}
\centering
 \begin{tikzpicture}[scale=.8]
     \draw[very thick] (1,0) coordinate (d0) -- (1.5,-2) coordinate (d1);
     \draw (1.1,-1.2) node {$1$};
	 \draw[fill] (d0) circle (3pt);
	 \draw[fill] (d1) circle (3pt);
     \draw (d0) -- (-2.5,-3) --  (-0.5,-3) -- (d0); 
     \draw (-.8,-2) node {$T_0$};
     \draw (d0) -- (3.5,-3) --  (5,-3) -- (d0); 
     \draw (3.4,-2) node {$T_{k+1}$};
     \draw (d1) -- (-1,-5) --  (.5,-5) -- (d1); 
     \draw (0.4,-4) node {$T_1$};
     \draw[very thick] (d1) -- (1.35,-3);
     \draw[fill] (1.5,-3) circle (.7pt);
     \draw[fill] (1.6,-3) circle (.7pt);
     \draw[fill] (1.7,-3) circle (.7pt);
     \draw[very thick] (d1) -- (1.85,-3);
     \draw (d1) -- (3,-5) --  (4.5,-5) -- (d1); 
     \draw (3,-4) node {$T_k$};
     \draw (-1.5,-3) node[below] {$P-1+q$};
     \draw (-.25,-5) node[below] {$P-1+u$};
     \draw (1.75,-5) node[below] {$\cdots$};
     \draw (3.75,-5) node[below] {$P-1+u$};
     \draw (4.25,-3) node[below] {$P-1+t$};
\end{tikzpicture}
      \caption{The decomposition of increasing compartmented trees from the second proof of Theorem~\ref{thm:Qplat_k}.}
      \label{fig:tree_decomp}
\end{figure}
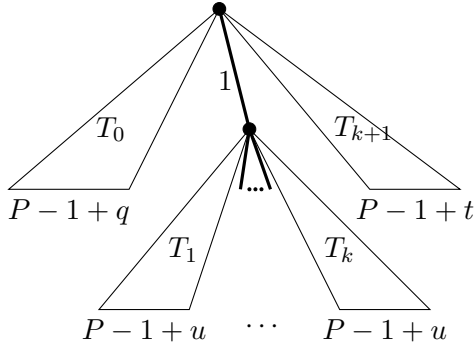

For $k=2$, the differential equation in Theorem~\ref{thm:Qplat_k} becomes simply
\begin{equation}\label{eq:Qplat} P'(z)=(P(z)-1+q)(P(z)-1+t)(P(z)-1+u). \end{equation}
The symmetry among the variables $q,t,u$ in this equation immediately implies that the joint distribution of the statistics $\asc$, $\des$ and $\plat$ on $\Q_n$ is symmetric, in the sense that it is invariant under any permutation of the three statistics, agreeing with \cite[Thm.\ 2.1]{janson_plane_2008}.

Setting $q=u=1$ and making the change of variable $z=\frac{z}{(1-t)^2}$ in Equation~\eqref{eq:Qplat} yields a different proof of \cite[Thm.\ 2.4]{gessel_stirling_1978}.

\section{Further directions}
\label{sec:further}

In this section we discuss some open problems and potential directions of further research.

\begin{problem}
Give a combinatorial proof of Theorem~\ref{thm:QQn}, in analogy to Gessel and Stanley's second proof of \cite[Thm.\ 2.1]{gessel_stirling_1978} for Stirling polynomials.
\end{problem}

In the statement of Theorem~\ref{thm:QQn}, the coefficient of $t^m$ in $\oQ_n(t)/(1-t)^{2n+1}$ can be interpreted as the number of permutations $\pi\in\QQ_n$ with $m$ bars inserted in some of the $2n+1$ spaces between the entries of $\pi$ (allowing bars to be inserted before the first entry and after the last entry), satisfying that there is some bar in each descent of $\pi$. A combinatorial proof would be obtained by showing that the number of such barred permutations equals $\frac{m^n}{n+1}\binom{m+n}{m}$.


Our second question asks for a bijective proof of Equation~\eqref{eq:QI}, which states that the number of quasi-Stirling permutations $\pi\in\QQ_n$ with $\des(\pi)=m+1$ equals the number of injections $[n-1]\to[2n]$ with $m$ excedances, for all $n$ and $m$ .

\begin{problem}
Describe a bijection $f:\QQ_n\to\In_{2n,n+1}$ such that $\des(\pi)-1=\exc(f(\pi))$ for all $\pi\in\QQ_n$.
\end{problem}


Viewing quasi-Stirling permutations as noncrossing matchings of $[2n]$ whose arcs are labeled with the labels $1,2,\dots,n$, it is reasonable to
consider the closely related set of labeled {\em nonnesting} matchings. These correspond to the set $\QQQ_n$ of permutations of $\{1,1,2,2,\dots,n,n\}$ that avoid $1221$ and $2112$. By definition, $\Q_n\subseteq\QQQ_n$. It is easy to see that $|\QQQ_n|=n!C_n=|\QQ_n|$, since unlabeled nonnesting matchings, just like their noncrossing counterparts, are also counted by the Catalan numbers (see \cite{chen_crossings_2007} for more information on crossings and nestings in matchings), and there are $n!$ ways to assign labels to the arcs. The distribution of the number of descents on $\QQQ_n$, which is different from its ditribution on $\QQ_n$, is the topic of a forthcoming preprint~\cite{archer_descents_nodate}.

Finally, it may be interesting to study the distribution on $k$-quasi-Stirling permutations of other statistics such as the number of inversions, the major index, the number of left-to-right minima, the number of inverse descents, the number of blocks of specified sizes, or the distance between occurrences of elements. The distribution of these statistics on Stirling and $k$-Stirling permutations has been studied in~\cite{park_r-multipermutations_1994,park_inverse_1994,janson_generalized_2011,kuba_analysis_2011}.

\subsection*{Acknlowledgements}
The author thanks Kassie Archer and Adam Gregory for making him aware of Conjecture~\ref{conj:archer} and the notion of quasi-Stirling permutations, as well as for helpful discussions.

\bibliographystyle{plain}
\bibliography{quasiStirling}

\end{document}